\documentclass[12pt,a4paper]{article}

\usepackage{amssymb}
\usepackage{amsmath, amsthm}
\usepackage{arydshln}
\usepackage{epsfig}
\usepackage{setspace}

\usepackage{geometry}

\def\RR{{\bf R}}
\def\ZZ{{\bf Z}}

\def\bigmid {\ \left|{\Large \strut}\right.}

\numberwithin{equation}{section}

\newtheorem{Thm}{Theorem}[section]

\newtheorem{Lem}[Thm]{Lemma}

\theoremstyle{definition}
\newtheorem*{Clm}{Claim}

\newtheorem{Ex}[Thm]{Example}

\title{Uniform modular lattices and affine buildings}
\author{Hiroshi HIRAI \\
Department of Mathematical Informatics, \\
Graduate School of Information Science and Technology,   \\
The University of Tokyo, Tokyo, 113-8656, Japan.\\
\texttt{\normalsize hirai@mist.i.u-tokyo.ac.jp}
}
\begin{document}

\maketitle
\begin{abstract}
In this paper, 
we present a simple lattice-theoretic characterization 
for affine buildings of type A.
We introduce a class of modular lattices, called uniform modular lattices, 
and show that uniform modular lattices and affine buildings of type A 
constitute the same object. 
This is an affine counterpart of 
the well-known equivalence between projective geometries ($\simeq$ complemented modular lattices) and spherical buildings of type A.
\end{abstract}

Keywords: Modular lattice, affine building of type A

\section{Introduction}

{\em Buildings}, due to Tits~\cite{Tits} and Bruhat and Tits~\cite{BruhatTits}, 
are simplicial complexes 
that extract combinatorial properties of algebraic groups, 
and have numerous important applications in branches of mathematics; see~\cite{BuildingBook,Scharlau}.
The present article addresses lattice-theoretic aspects of buildings.
As is well-known, spherical buildings of type A 
and (generalized) projective geometries
are the same mathematical object~\cite{Tits}: 
All chains (flags) of subspaces in a projective geometry form 
a spherical building of type A, 
and any spherical building of type A is obtained in this way.  
In lattice theory, a classical result~\cite{Birkhoff} by Birkhoff
says that the subspace lattice of a projective geometry is 
exactly a {\em complemented modular lattice} of finite rank 
(also known as a {\em modular geometric lattice})---a modular lattice 
in which the maximum element is the join of atoms.
Thus we can say:
\begin{Thm}[{\cite{Tits}; see e.g., \cite[THEOREM 4.1.4]{Scharlau}}]\label{thm:main0}
	Complemented modular lattices of finite rank and spherical buildings of type A 
	constitute the same object.
\end{Thm}

The goal of this paper is to establish 
an analogue of this theorem for affine buildings of type A. 
We introduce an affine analogue of complemented modular lattices, 
named {\em uniform modular lattices}.
This class of modular lattices is simply defined:
A modular lattice ${\cal L}$ is called {\em uniform} if
the operator $x \mapsto$ (the join of all elements covering $x$)
is an automorphism on ${\cal L}$. 
Our main result, which might be a reasonable affine counterpart of Theorem~\ref{thm:main0},  
is as follows.
\begin{Thm}\label{thm:main}
Uniform modular lattices and 	
affine buildings of type A constitute the same object.
\end{Thm}
The precise meaning of this theorem is 
explained in Theorems~\ref{thm:main1} and \ref{thm:main2}:
The former says that  
any uniform modular lattice ${\cal L}$ gives rise to 
an affine building of type A as the projection ${\cal C}({\cal L})$ of a subcomplex of its order complex,
and the latter constructs a uniform modular lattice ${\cal L}(\varDelta)$ 
from any affine building $\varDelta$ of type A 
for which ${\cal C}({\cal L}(\varDelta)) = \varDelta$.

By relaxing modularity to semimodularity, 
we obtain, in the same way, a further natural class of lattices, called {\em uniform semimodular lattices}. 
This class of lattices is studied in the companion paper~\cite{HH18a}.

After submitting this paper, 
we found a closely related approach by Abels~\cite{Abels91}. 
He introduced the notion of 
{\em semimodular lattices with cofinal $Z$-action}, and
studied, in a lattice-theoretic way, 
the gallery distance in the affine building of ${\rm SL}(K^n)$ 
for a field $K$ with a discrete valuation.
He noticed that the affine building of ${\rm SL}(K^n)$
gives rise to a modular lattice with cofinal $Z$-action 
and that the simplicial structure of the building is described 
by this lattice.
A uniform (semi)modular lattice ${\cal L}$
is indeed an example of (semi)modular lattices with cofinal $Z$-action, 
where
the constructions of lattice ${\cal L}$ from the affine building of ${\rm SL}(K^n)$ 
and the simplicial complex ${\cal C}({\cal L})$ from ${\cal L}$
coincide with that in \cite{Abels91}; see Example~\ref{ex:module}.   
Thus one can also say that the result in this paper pushes 
the lattice-theoretic approach of \cite{Abels91}
toward general affine buildings of type A.

The rest of this paper is organized as follows.
In Section~\ref{sec:pre}, we introduce basic terminologies 
and results on lattices and buildings.
In Section~\ref{sec:uniform}, 
we formally introduce the concept of uniform modular lattices, and
establish the main theorem (Theorem~\ref{thm:main}) 
by proving Theorems~\ref{thm:main1} and \ref{thm:main2}. 
In Section~\ref{sec:concluding}, we give some remarks 
that include future applications of 
uniform modular lattices to combinatorial optimization.

\section{Preliminaries}\label{sec:pre}
Our references for lattices are \cite[Chapter II]{Aigner}, the first edition of \cite{Birkhoff}, and \cite{Gratzer}. 
The references for buildings are \cite{BuildingBook,Garrett,Scharlau}.
\subsection{Basic notation}
Let $\ZZ$ and $\RR$ be the set of integers and reals, respectively.
Let $\RR_+$ denote the set of nonnegative reals.
In $\RR^n$, let $e_i$ denote the $i$th unit vector, 
${\bf 0}$ the zero vector, 
and ${\bf 1}$ the all-one vector.
For $x, y \in \RR^n$, by $x \leq y$ we mean that $x_i \leq y_i$ for all $i=1,2,\ldots,n$.
Let $\min (x,y)$ and $\max (x,y)$ 
be defined as the vectors in $\RR^n$ obtained from $x,y$
by taking minimum and maximum componentwise, respectively. 

\paragraph{Lattices.}
We will use the basic terminology of poset and lattice.
A poset (partially ordered set) ${\cal L}$ is a set endowed 
with a partial order relation $\preceq$, 
where $x \prec y$ is meant as $x \preceq y$ and $x \neq y$.
The {\em opposite} $\check{\cal L}$ of ${\cal L}$ 
is the poset on ${\cal L}$ with partial order $\preceq'$ defined by 
$x \preceq' y \Leftrightarrow y \preceq x$. 
The partial order of the direct product ${\cal L} \times {\cal L}'$ of two posets ${\cal L},{\cal L}'$
is defined as $(x,x') \preceq (y,y')$ $\Leftrightarrow$ $x \preceq y$ and $x' \preceq y'$.

For elements $x,y$ with $x \preceq y$, the {\em interval} of $x,y$ 
is the set of elements $z$ with $x \preceq z \preceq y$, and is denoted by $[x,y]$.
We say that $y$ {\em covers} $x$ if $x \neq y$ and $[x,y] = \{x,y\}$.
A totally ordered subset ${\cal C}$ of ${\cal L}$ is called a {\em chain}.
If a chain ${\cal C}$ consists of $x_0,x_1,\ldots,x_m,\ldots$ 
with $x_i \prec x_{i+1}$ for all indices $i$, then ${\cal C}$ is denoted by
$(x_0 \prec x_1 \prec \cdots \prec x_m \prec \cdots)$.
The {\em length} of chain ${\cal C}$ is defined as its cardinality $|{\cal C}|$ minus one.
The unique minimal common upper bound of elements $x,y$
is called the {\em join} of $x,y$, and is denoted by $x \vee y$ if it exists.
The unique maximal common lower bound of $x,y$
is called the {\em meet} of $x,y$, and is denoted by $x \wedge y$ if it exists.
A {\em lattice} ${\cal L}$ is a poset such that 
every pair of elements has the join and meet.
An {\em isomorphism} between two lattices ${\cal L}$ and ${\cal L'}$ 
is a bijection $\varphi:{\cal L} \to {\cal L}'$
such that $\varphi(x \wedge y) = \varphi(x) \wedge \varphi(y)$
and $\varphi(x \vee y) = \varphi(x) \vee \varphi(y)$
for all $x,y \in {\cal L}$, or equivalently, 
$\varphi(x) \preceq \varphi(y) \Leftrightarrow x \preceq y$ for all $x,y \in {\cal L}$.
In addition, if ${\cal L} = {\cal L'}$, then $\varphi$ is called an {\em automorphism} 
on ${\cal L}$.
For a subset $S \subseteq {\cal L}$, 
the unique maximal lower bound of $S$ (the {\em meet} of $S$) is denoted by 
$\bigwedge S$ if it exists, and the unique minimal upper bound of $S$ (the {\em join} of $S$) is denoted by $\bigvee S$ if it exists.
In a lattice ${\cal L}$, the minimum element $\bigwedge {\cal L}$, if it exists,
is denoted by $\bar 0$, and the maximum element $\bigvee {\cal L}$, if it exists,
is denoted by $\bar 1$.
In a lattice ${\cal L}$ having the minimum element $\bar 0$, 
the rank $r(x)$ of element $x$ is the maximum length of a chain in $[\bar 0,x]$.
The rank of ${\cal L}$ (having $\bar 0$ and $\bar 1$) is 
the maximum length of a maximal chain of ${\cal L}$.
By an {\em atom} we mean an element of rank one.
A {\em sublattice} of a lattice ${\cal L}$ is a subset ${\cal L'} \subseteq {\cal L}$
with the property that $x,y \in {\cal L'}$ imply $x \wedge y, x \vee y \in {\cal L}'$.
Intervals are sublattices.
In this paper, any lattice satisfies the following finiteness assumption:
\begin{itemize}
	\item[(F)] Any interval $[x,y]$ has a finite rank $r[x,y] < \infty$.
\end{itemize}

\paragraph{Simplicial complexes.}
A {\em simplicial complex} $\varDelta$ is a family of subsets of a nonempty set $V$ such that
$A' \subseteq A \in \varDelta$ implies $A' \in \varDelta$.
A member $A$ in $\varDelta$ is called a {\em simplex}, and 
its {\em dimension} is defined as $|A|-1$. 
The dimension of $\varDelta$ is defined as 
the maximum dimension of a simplex in $\varDelta$.
A $0$-dimensional simplex is called a {\em vertex}.
The set of vertices is denoted by $\varDelta^0$.
We assume that all singleton $\{v\}$ $(v \in V)$ are vertices, and hence
$\varDelta^0$ is identified with $V$. 
Two simplicial complexes 
$\varDelta, \varDelta'$ are {\em isomorphic} 
if there exists a bijection $\rho: \varDelta^0 \to \varDelta'^0$, called 
an {\em isomorphism}, such that $A \in \varDelta \Leftrightarrow \rho(A) \in \varDelta'$.
An isomorphism $\rho$ induces an inclusion-preserving bijection 
$\varDelta \to \varDelta'$ by $A \mapsto \rho(A)$; 
therefore $\rho$ is also regarded as $\varSigma \to \varSigma'$.

For a poset ${\cal P}$, the {\em order complex} ${\cal O}({\cal P})$ of ${\cal P}$ is 
the simplicial complex on ${\cal P}$ consisting of all chains of finite length.

A {\em geometric realization} $|\varDelta|$ of $\varDelta$ is 
the set of all functions $u : \varDelta^0 \to \RR_+$ 
such that 
$\{x \in \varDelta^0 \mid u(x) > 0\} \in \varDelta$ and $\sum_{x \in \varDelta^0} u(x) =1$.
Then, abstract simplices in $\varDelta$ become geometric simplices in $|\varDelta|$ with mutually disjoint relative interiors.
%

\subsection{Modular lattices}\label{subsec:modular}

A lattice ${\cal L}$ is called {\em modular} 
if $(y \wedge z) \vee x = (x \vee z) \wedge y$
for all triples $x,y,z \in {\cal L}$ with $x \preceq y$. 
Modular lattices satisfy the Jordan-Dedekind chain condition:
\begin{itemize}
	\item[(JD)] Maximal chains in any interval have the same length.
\end{itemize}
A {\em valuation} of a lattice ${\cal L}$
is a function $v:{\cal L} \to \RR$ satisfying $v(x) < v(y)$ for all 
$x,y \in {\cal L}$ with $x \prec y$, and 
\begin{equation}\label{eqn:valuation}
v(x) + v(y) = v(x \wedge y) + v(x \vee y) \quad (x,y \in {\cal L}).
\end{equation}
It is well-known that the rank function of a modular lattice having $\bar 0$ is a valuation; 
see e.g., \cite[Theorem 2.27]{Aigner}.
Conversely the existence of a valuation implies the modularity.
\begin{Lem}[{see \cite[Theorem 3.11]{Birkhoff}}]\label{lem:valuation}
	A lattice ${\cal L}$ having a valuation is a modular lattice.
\end{Lem}
\begin{proof}[Sketch of proof]
    For $x,y,z \in {\cal L}$ with $x \preceq y$, it always holds $(y \wedge z) \vee x \preceq (x \vee z) \wedge y$. If ${\cal L}$ has a valuation $v$, 
    by using (\ref{eqn:valuation}), one can see that $v((y \wedge z) \vee x) 
    = v((x \vee z) \wedge y)$, which implies $(y \wedge z) \vee x = (x \vee z) \wedge y$.
\end{proof}
This proof uses the following obvious rank-comparison argument, which will be often used later:
\begin{itemize}
	\item[(rc)] $x \preceq y$ and $v(x) = v(y)$ imply $x = y$. 
\end{itemize}
For a subset $S$ of lattice ${\cal L}$, let $\langle S \rangle$ 
denote the sublattice of ${\cal L}$ generated by $S$ 
($=$ the minimum sublattice containing $S$).
\begin{Lem}[{See \cite[Theorems 348, 364]{Gratzer}}]\label{lem:isomorphism}
	Let ${\cal L}$ be a modular lattice.
	For $p,q \in {\cal L}$, the following hold:
	\begin{itemize}
		\item[{\rm (1)}] $[p,p \vee q]$ 
		is isomorphic to $[p \wedge q,q]$ by $x \mapsto x \wedge q$ with inverse $y \mapsto y \vee p$.
		\item[{\rm (2)}] $\langle [p \wedge q,p] \cup [p \wedge q,q] \rangle$ is isomorphic to $[p \wedge q,p] \times [p \wedge q,q]$
		by $x \mapsto (x \wedge p, x \wedge q)$ with inverse $(u,v) \mapsto u \vee v$.
	\end{itemize}	 
\end{Lem}

In a lattice ${\cal L}$ with $\bar 0$ and $\bar 1$, 
a {\em complement} of an element $x \in {\cal L}$ is an element $y$ 
such that $x \wedge y = \bar 0$ and $x \vee y = \bar 1$.
A lattice ${\cal L}$ with $\bar 0$ and $\bar 1$ is said to be {\em complemented} 
if every element has a complement. 
The following facts are basic; see e.g., \cite[Theorem 4.1]{Birkhoff}.
\begin{itemize}
	\item[(cm1)] A modular lattice is complemented if and only if $\bar 1$ is the join of atoms.
	\item[(cm2)] Every interval of a complemented modular lattice is complemented modular.
	\item[(cm3)] The opposite of a complemented modular lattice is complemented modular.
\end{itemize}
Note that (cm3) is immediate from the definition, 
and (cm1) is true under the assumption (F); 
in particular, {\em modular geometric lattices} (see \cite[Section II. 3]{Aigner})  and 
complemented modular lattices are the same in this paper.
%
%

In a complemented modular lattice ${\cal L}$ of rank $n$, 
a set of $k$ atoms $a_1,a_2,\ldots,a_k$ 
is said to be {\em independent} if $r(a_1 \vee a_2 \vee \cdots \vee a_k) = k$, or equivalently 
if $a_i \wedge \bigvee_{j \neq i} a_j = \bar 0$ for all $i$.
A {\em basis} of ${\cal L}$ is a set of $n$ independent atoms. 
\begin{Ex}\label{ex:01} 
	The partial order $\preceq$ on $\{0,1\}^n$ is defined as 
	the vector order $\leq$.  
	Then $\{0,1\}^n$ is a complemented modular lattice of rank $n$ 
	(more precisely it is a Boolean lattice).
	The meet and join are given by
	$u \wedge v = \min (u,v)$
	and $u \vee v = \max (u,v)$. 
	Unit vectors $e_1,e_2,\ldots,e_n$ form the unique basis of $\{0,1\}^n$.
\end{Ex}

\begin{Ex}
	Let ${\cal L}$ be the family of all vector subspaces $X$ of a vector space $V$ 
	of dimension $n$. Regard ${\cal L}$ as a poset with respect to inclusion order $\subseteq$.
	Then ${\cal L}$ is a complemented modular lattice of rank $n$, 
	where $\wedge$ and $\vee$ are equal to $\cap$ and $+$, respectively, 
	and $X \mapsto \dim X$ 
	is a valuation (and the rank function). 
	A basis of ${\cal L}$ is precisely the set of 
	$1$-dimensional vector spaces corresponding to a basis of $V$.
\end{Ex}

In the following three lemmas, 
${\cal L}$ is assumed to be 
a complemented modular lattice of rank $n$.  
\begin{Lem}\label{lem:complement}
Let $C$ be a maximal chain in ${\cal L}$, and $p$ an element in ${\cal L}$.
	\begin{itemize}
		\item[{\rm (1)}] There exists a complement $q$ of $p$ in ${\cal L}$
		such that the sublattice $\langle [\bar 0,p] \cup [\bar 0,q] \rangle$ contains $C$.
		\item[{\rm (2)}] In addition,
		if a basis $a_1,a_2,\ldots,a_k$ of 
		$[\bar 0,p]$ generates $C \wedge p$ and a basis $b_1,b_2,\ldots, b_l$ 
		of $[p,\bar 1]$ generates $C \vee p$, 
		then $a_1,a_2,\ldots, a_k,b_1 \wedge q, b_2 \wedge q,\ldots,b_l \wedge q$
		is a basis of ${\cal L}$ that
		generates $C$.
	\end{itemize}
\end{Lem}
\begin{proof}
	(1). We use induction on the rank $n$ of ${\cal L}$. 
	Suppose $C = (\bar 0 = x_0 \prec x_1 \prec \cdots \prec x_n = \bar 1)$.
	We may assume that $n \geq 2$.
	
	Case 1: $p \preceq x_{n-1}$.
	Consider interval $[\bar 0, x_{n-1}]$, which is complemented modular by (cm2), 
	and consider maximal chain $C' = (\bar 0 = x_0 \prec x_1 \prec \cdots \prec x_{n-1})$ in $[\bar 0, x_{n-1}]$.
	By induction, there is a complement $q' \in [\bar 0,x_{n-1}]$ of $p$ 
	such that $\langle [\bar 0,p] \cup [\bar 0,q'] \rangle$ contains $C'$.
	Now $q' \preceq x_{n-1} \prec x_n = \bar 1$. 
	We can choose a complement $q \in [q', \bar 1]$ of $x_{n-1}$ (which covers $q'$).
	Then $q$ is a desired complement of $p$ in ${\cal L}$.
	Indeed, $p \vee q = p \vee q' \vee q = x_{n-1} \vee q = \bar 1$.
	Also $p \wedge q = {\bar 0}$ follows 
	from (rc) and $r(q) - r(\bar 0) = 1 + r(q') - r(\bar 0) = 1 + r(x_{n-1}) - r(p) = r(\bar 1) - r(p) = r(q) - r(p \wedge q)$ (implying $r(\bar 0) = r(p \wedge q)$).

	Case 2: $p \not \preceq x_{n-1}$.
	Consider $p' := x_{n-1} \wedge p$.
	Then $p'$ is covered by $p$; consider 
	(\ref{eqn:valuation}) for the setting $v =r$, $x = x_{n-1}$, and $y =p$.
	As above, 
	consider complemented modular lattice $[\bar 0, x_{n-1}]$ and
	maximal chain $C'$.
	By induction, there is a complement $q \in [\bar 0,x_{n-1}]$ of $p'$ such that 
	$\langle [\bar 0,p'] \cup [\bar 0,q] \rangle$ contains $C'$.
	Since $p \vee q = \bar 1$, 
	the sublattice $\langle [\bar 0,p] \cup [\bar 0,q] \rangle$ contains $C$.

	(2). Consider $x \in C$. Then 
	$x \wedge p$ is the join of a subset of $a_1,a_2,\ldots,a_k$, 
	say $x \wedge p = a_1 \vee a_2 \vee \cdots \vee a_{k'}$.
	Also $x \vee p$ is represented as $x \vee p = b_1 \vee b_2 \vee \cdots \vee b_{l'}$.
	Now $x = (x \wedge p) \vee (x \wedge q) = (x \vee p) \wedge (x \vee q)$ holds by 
	Lemma~\ref{lem:isomorphism}~(2) (applied to the opposite).
	From $x = (x \vee p) \wedge (x \vee q)$, 
	we have $x \wedge q = (x \vee p) \wedge (x \vee q) \wedge q = (x \vee p) \wedge q$.
	Thus $x = (x \wedge p) \vee ((x \vee p) \wedge q) 
	= a_1 \vee a_2 \vee \cdots \vee a_{k'} \vee ((b_1 \vee b_2 \vee \cdots \vee b_{l'}) \wedge q) = a_1 \vee a_2 \vee \cdots \vee a_{k'} \vee (b_1 \wedge q) \vee (b_2 \wedge q) \vee \cdots \vee (b_{l'} \vee q)$, where the last equality follows from Lemma~\ref{lem:isomorphism}~(1) and $b_i \succeq p$.
	Therefore $a_1,a_2,\ldots, a_k,b_1 \wedge q, b_2 \wedge q,\ldots,b_l \wedge q$ generate $C$. They form a basis 
	since their join equals $\bar 1$ (by Lemma~\ref{lem:isomorphism}) with $k+l = n$. 
\end{proof}	

The sublattice $\langle a_1,a_2,\ldots,a_n \rangle$ 
generated by a basis $a_1,a_2,\ldots,a_n$
is isomorphic to Boolean lattice $\{0,1\}^n$ by 
$e_{i_1} + e_{i_2} + \cdots + e_{i_k} \mapsto 
a_{i_1} \vee a_{i_2} \vee \cdots \vee a_{i_k}$.
We call such a sublattice a {\em $\{0,1\}^n$-skeleton}.
The next lemma is a folklore in theory of modular lattice, and is
essentially one of axioms of building; see (B2) in Section~\ref{subsec:building}.
\begin{Lem}[{See \cite[Theorem 363]{Gratzer}}]\label{lem:twochains}
	For two chains $C,D$ in ${\cal L}$, 
	there exists a $\{0,1\}^n$-skeleton in ${\cal L}$
	containing $C$ and $D$.
\end{Lem}
\begin{proof}
	We use induction on $n$; the case of $n=1$ is obvious.
	Thus $n \geq 2$.
	We may assume that 
	$C = (\bar 0 = x_0 \prec x_1 \prec \cdots \prec x_n = \bar 1)$ and
	$D = (\bar 0 = y_0 \prec y_1 \prec \cdots \prec y_n = \bar 1)$.
	Consider complemented modular lattice $[\bar 0, x_{n-1}]$, and
	maximal chains $C' := (\bar 0 = x_0 \prec x_1 \prec \cdots \prec x_{n-1})$ 
	and $D' := (\bar 0 = y_0 \wedge x_{n-1}  \preceq y_1 \wedge x_{n-1} 
	\preceq \cdots \preceq y_{n} \wedge x_{n-1} = x_{n-1})$.
	By induction, there is a $\{0,1\}^{n-1}$-skeleton ${\cal F}' = \langle a_1,a_2,\ldots,a_{n-1} \rangle$ in $[\bar 0, x_{n-1}]$
	containing $C',D'$.
	By the previous lemma, 
	we can choose a complement $a$ of $x_{n-1}$ in ${\cal L}$
	such that $\langle [\bar 0,x_{n-1}] \cup [\bar 0,a]\rangle$ contains $D$.
	Here $a$ is an atom of ${\cal L}$, and
	$a_1,a_2,\ldots,a_{n-1}, a$ form a basis of ${\cal L}$.
	Thus the $\{0,1\}^n$-skeleton ${\cal F} := \langle a_1,a_2,\ldots,a_{n-1},a \rangle$
	contains $C$ and $D$, as required.
\end{proof}

Let
$C = (\bar 0=x_0 \prec x_1 \prec \cdots \prec x_n= \bar 1)$ be a maximal chain in ${\cal L}$.
The {\em relative position} $y_C \in \{0,1\}^n$ of $y \in {\cal L}$ with respect to $C$ 
is defined as follows.
Choose a $\{0,1\}^n$-skeleton ${\cal F}$ containing $C,y$ via Lemma~\ref{lem:twochains}. 
Regard ${\cal F}$ as $\{0,1\}^n$, 
where we assume $x_i = x_{i-1} + e_i$ for $i=1,2,\ldots,n$ by relabeling.
Define the relative position $y_C \in \{0,1\}^n$ 
as the 0,1-vector corresponding to $y$ in this coordinate ${\cal F}$.
\begin{Lem}\label{lem:indep}
	The relative position $y_C$ of 
	$y$ is independent of 
	the choice of a $\{0,1\}^n$-skeleton containing $C,y$.
\end{Lem}
\begin{proof}
	$y_C$ is exactly the sum of unit vectors $e_i$ 
	over indices $i$ with $e_i \leq y (= y_C)$. 
	Here $e_i \leq y$ ($\Leftrightarrow e_i \preceq y$) 
	is equivalent to the lattice condition 
	$x_i \wedge y \succ x_{i-1} \wedge y$ ($\Leftrightarrow$ $\min (x_i,y) - \min (x_{i-1},y) = e_i$), 
	which is independent of the $\{0,1\}^n$-skeleton.
\end{proof}

\subsection{Spherical/affine buildings of type A}\label{subsec:building}
%
We first introduce the spherical/affine Coxeter complex of type A 
(without group-theoretic language).
We consider the decomposition of $\RR^n/\RR {\bf 1}$ 
(the quotient space of $\RR^n$ by $\RR {\bf 1}$)
by the following hyperplanes in $\RR^n$:
\begin{equation}
H_{ij} := \{ x \in \RR^n \mid x_i = x_j \} \quad (1 \leq i < j \leq n).
\end{equation}
The closure of each connected component of 
$(\RR^n \setminus \bigcup_{1 \leq i<j \leq n} H_{ij}) / \RR {\bf 1}$ is a simplicial cone that is the conical hull of $n-1$ vectors 
\begin{equation}\label{eqn:Coxeter_s}
e_{\sigma(1)},\ e_{\sigma(1)} + e_{\sigma(2)},\ 
\ldots,\  e_{\sigma(1)} + e_{\sigma(2)} + \cdots + e_{\sigma(n-1)}
\end{equation}
modulo $\RR {\bf 1}$
for a permutation $\sigma$ on $\{1,2,\ldots,n\}$.
The {\em spherical Coxeter complex of type A} is the simplicial complex
on $\{0,1\}^n \setminus \{{\bf 0}, {\bf 1}\}$
whose maximal simplices have vertices of form (\ref{eqn:Coxeter_s}). 

Next we introduce the affine Coxeter complex of type A.
Consider the decomposition of $\RR^n/\RR {\bf 1}$ by 
the following affine hyperplanes:
\begin{equation}
H_{ij,k} := \{ x \in \RR^n \mid x_i = x_j + k\} \quad (1 \leq i < j \leq n, k \in \ZZ).
\end{equation}
The closure of each connected component of 
$(\RR^n \setminus \bigcup_{1 \leq i<j \leq n, k \in \ZZ} H_{ij,k}) / \RR {\bf 1}$ is a simplex that is the convex hull of $n$ vertices
\begin{equation}\label{eqn:Coxeter_e}
z,\ z + e_{\sigma(1)},\ z+ e_{\sigma(1)} + e_{\sigma(2)},\ \ldots,\ z + e_{\sigma(1)} + e_{\sigma(2)} + \cdots + e_{\sigma(n-1)}
\end{equation}
modulo $\RR {\bf 1}$ for a permutation $\sigma$ on $\{1,2,\ldots,n\}$ 
and $z \in \ZZ^n$.
The {\em affine Coxeter complex of type A} 
is the simplicial complex on $\ZZ^n / \RR {\bf 1}$ 
whose maximal simplices have vertices of form (\ref{eqn:Coxeter_e}).  

A {\em spherical/affine building of type A} 
is a simplicial complex $\varDelta$ having a family of 
subcomplexes, called {\em apartments}, satisfying the following axiom:
\begin{itemize}
	\item[(B1)] Each apartment is isomorphic 
	to the spherical/affine Coxeter complex of type~A.
	\item[(B2)] For two simplices $A,B \in \varDelta$, 
	there is an apartment $\varSigma$ containing $A,B$.
	\item[(B3)] If two apartments $\varSigma, \varSigma'$ 
	contain simplices $A,B$, then there is an isomorphism $\varphi: \varSigma \to \varSigma'$ fixing $A$ and $B$ pointwise, 
	i.e., $\varphi(x) = x$ for $x \in A \cup B$.
\end{itemize}
The definition of general spherical/affine buildings is 
obtained by using general spherical/affine Coxeter complex in axiom (B1). 
The basic properties of buildings that we will use 
are summarized as follows; see~\cite[Chapter 4]{BuildingBook}.
They hold for general spherical/affine buildings, not restricted to type~A.
\begin{itemize}
	\item[(col)] A spherical/affine building (of type A) with dimension $n$
	admits a labeling $\ell: \varDelta^0 \to \{0, 1,2,\ldots,n\}$, called a {\em coloring}, 
	with the property that any distinct vertices $x,y$ 
	in any simplex have distinct colors $\ell(x) \neq \ell(y)$.
	A coloring is automatically determined from 
	any coloring (i.e., bijection to $\{0,1,\ldots,n\}$) of any maximal simplex,
	Moreover any two colorings $\ell,\ell'$ are equivalent in the sense that  $\ell' = \kappa \circ \ell$ holds
	for some bijection $\kappa$ on $\{0, 1,2,\ldots,n\}$.
	\item[(B3$^+$)] The isomorphism $\varphi$ in (B3) can be taken to be color-preserving, 
	i.e., $\ell(\varphi(x)) = \ell(x)$.
	\item[(ret)] For an apartment $\varSigma$ and a maximal simplex $A$ in $\varSigma$, 
	define a map $\rho_{\varSigma,A}: \varDelta \to \varSigma$ as follows:
	For $B \in \varDelta$, choose an apartment $\varSigma'$ containing $A$ and $B$ via (B2), 
	an isomorphism $\varphi: \varSigma' \to \varSigma$ fixing $A$ via (B3), 
	and let $\rho_{\varSigma, A}(B) := \varphi(B)$.
	This map $\rho_{\varSigma,A}$ is independent of the choice 
	of an apartment $\varSigma'$, and is a (color-preserving) retraction to $\varSigma$, i.e., $\rho_{\varSigma,A}(B) = B$ for $B \in \varSigma$. The map $\rho_{\varSigma,A}$ 
	is called the {\em canonical retraction}.
\end{itemize}

The geometric realization $|\varDelta|$ 
of an affine building $\varDelta$ (of type A) 
admits a natural ``Euclidean" metric; see \cite[Chapter 11]{BuildingBook}.
As seen above, 
the affine Coxeter complex $\varSigma$ 
is a triangulation of $\RR^n / \RR {\bf 1}$, 
and the geometric realization $|\varSigma|$ 
is naturally regarded as $\RR^n / \RR {\bf 1}$.
Define a Euclidean metric $d_\varSigma$ on $|\varSigma| = \RR^n / \RR {\bf 1}$
by $d_\varSigma(x + \RR {\bf 1},y + \RR {\bf 1}) := \|\bar x - \bar y \|_2$, 
where $\bar x$ is the unique vector with $\bar x - x \in \RR {\bf 1}$ 
and $\sum_{i=1}^n \bar x_i = 0$.
For two points $x,y$ in the geometric realization $|\varDelta|$ of an affine building $\varDelta$, define $d(x,y) := d_{\varSigma}(x,y)$ 
by choosing an apartment $\varSigma$ with $x,y \in |\varSigma|$ (via (B2)).
In fact, $d(x,y)$ is independent of the choice of an apartment $\varSigma$.
Hence $d$ is a well-defined distance function, and 
$|\varDelta|$ becomes a metric space.
The metric space $|\varDelta|$ has a nice property on geodesics, where
a {\em geodesic} between two point $x,y$ is 
a path $\gamma: [0,1] \to |\varDelta|$ with $\gamma(0) =x$, $\gamma(1) = y$, and 
$d(\gamma(s), \gamma(t)) = |\gamma(s) - \gamma(t)|d(x,y)$
for $s,t \in [0,1]$.
\begin{itemize}
	\item[(geo)] $|\varDelta|$ is uniquely geodesic, that is,  
	there is a unique geodesic between any two points $x,y \in |\varDelta|$.
\end{itemize}
This property is a consequence of the fact that $|\varDelta|$ is a  
{\em CAT(0) space}; see \cite{BuildingBook,BrHa}.

In the following, we explain 
the relationship (of Theorem~\ref{thm:main0}) between complemented modular lattices 
and spherical buildings of type A.
We here provide a larger part of the proof, 
since there seems no reference including such a direct proof  
without group-theory and incidence-geometry arguments, 
and the proof of Theorem~\ref{thm:main} goes completely in parallel.

Notice that the spherical Coxeter complex of type A 
is nothing but the order complex of  
poset $\{0,1\}^n \setminus \{{\bf 0}, {\bf 1}\}$.
\begin{Thm}[\cite{Tits}]\label{thm:spherical-1}
	Let ${\cal L}$ be a complemented modular lattice of rank $n \geq 3$.
	Then the order complex ${\cal O}({\cal L} \setminus \{\bar 0,\bar 1\})$ 
	is a spherical building of type A with dimension $n-2$.
\end{Thm}
\begin{proof}
	We show that subcomplexes of ${\cal O}({\cal L})$ induced by  
	$\{0,1\}^n$-skeletons (deleted by ${\bf 0},{\bf 1}$) satisfy the axiom of apartment.  
	They are obviously isomorphic to the spherical Coxeter complex of type A, implying~(B1).
	Consider two maximal simplices $A,B$, where $A \cup \{\bar 0,\bar 1\}$ and $B \cup \{\bar 0,\bar 1\}$ 
	are maximal chains in ${\cal L}$.
	By Lemma~\ref{lem:twochains}, there is a $\{0,1\}^n$-skeleton containing $A,B$. 
	This implies (B2).
	Suppose that two $\{0,1\}^n$-skeletons ${\cal F},{\cal G}$ contain $A,B$. 
	Suppose further that ${\cal F} = \langle f_1,f_2,\ldots,f_n \rangle$
	and ${\cal G} = \langle g_1,g_2,\ldots,g_n \rangle$ for bases $f_1,f_2,\ldots,f_n$ 
	and $g_1,g_2,\ldots,g_n$ of ${\cal L}$.
	By renumbering,  
	we can assume that 
	$A = \{ g_1 \vee g_2 \vee \cdots \vee g_k \}_{k=1,2,\ldots,n-1} = \{ f_1 \vee f_2 \vee \cdots \vee f_k \}_{k=1,2,\ldots,n-1}$. 
	Define $\varphi:{\cal F} \to {\cal G}$ 
	by $f_{i_1} \vee f_{i_2} \vee \cdots \vee f_{i_k} \mapsto g_{i_1} \vee g_{i_2} \vee \cdots \vee g_{i_k}$, which obviously induces an isomorphism between 
    ${\cal O}({\cal F})$ and ${\cal O}({\cal G})$.
	Also $\varphi(x) = x$ for $x \in A$.
	For $x \in {\cal F} \cap {\cal G}$, 
	if $x = f_{i_1} \vee f_{i_2} \vee \cdots \vee f_{i_k}$, i.e., 
	$e_{i_1} + e_{i_2} + \cdots + e_{i_k}$ is the relative position of $x$ with respect to $A \cup \{\bar 0, \bar 1\}$, 
	then $x = g_{i_1} \vee g_{i_2} \vee \cdots \vee g_{i_k}$ holds by Lemma~\ref{lem:indep}, 
	and hence $\varphi(x) = x$.
	Thus $\varphi$ is the identity on ${\cal F} \cap {\cal G}$, and consequently
	fixes $B$ pointwise, implying~(B3).
\end{proof}

Next we construct a complemented modular lattice 
from a spherical building $\varDelta$ of type A.
Our construction uses a special coloring; see (col) for colorings.
A {\em natural coloring} is a coloring $\ell$ with the property that
for every apartment $\varSigma$ there 
is an isomorphism $\rho :\varSigma \to {\cal O}(\{0,1\}^n \setminus \{{\bf 0}, {\bf 1}\})$ 
with $\ell(x) = \sum_{i=1}^n \rho(x)_i$ for all $x \in \varSigma^0$.
In fact, a natural coloring coincides 
with the {\em natural ordering} in the sense of~\cite{Tits}.
In general, 
a coloring is obtained, in group-theoretic way, 
by associating each vertex with one of generators of the Coxeter 
group corresponding to the Coxeter complex.
A natural ordering is the ordering of the generators
so that consecutive numbers are assigned to adjacent generators 
in the Dynkin diagram of type A (that is a path).
\begin{Lem}
	A natural coloring exists.
\end{Lem}
\begin{proof}[Sketch of proof]
Consider an arbitrary apartment $\varSigma$.
Then $\varSigma$ is isomorphic to the order complex 
${\cal O}(\{0,1\}^n \setminus \{{\bf 0},{\bf 1}\})$. 
Identify $\varSigma^0$ with $\{0,1\}^n \setminus \{{\bf 0},{\bf 1}\}$.
Define the color $\ell(x)$ of $x \in \varSigma^0$ as $\sum_{i=1}^n x_i$.
This is a coloring of $\varSigma$.
Extend this coloring to the whole coloring $\ell$ on $\varSigma^0$ (via (col)).
In fact, $\ell$ is a natural coloring.
One can see this fact by counting and comparing the numbers of 
neighbors of a vertex with respect to their color. 
We will do this for the affine case in the proof of Lemma~\ref{lem:natural}.
The proof goes along precisely the same way.
\end{proof}
Fix an arbitrary natural coloring $\ell$.
Define a partial order $\preceq$ on $\varDelta^0$ by 
$x \preceq y$ if $x$ and $y$ belong to a common simplex and $\ell(x) \leq \ell(y)$.
It turns out in the next proposition that $\preceq$ is a partial order.
Add a minimum element $\bar 0$ and a maximal element $\bar 1$ to $\varDelta^0$.
The resulting poset is denoted by ${\cal L}(\varDelta)$. 
\begin{Thm}[\cite{Tits}]
	Let $\varDelta$ be a spherical building of type A.
	Then ${\cal L}(\varDelta)$ is a complemented modular lattice, 
	where $\varDelta$ is equal to the order complex of 
	${\cal L}(\varDelta) \setminus \{\bar 0,\bar 1\}$.
\end{Thm}
\begin{proof}	
	We first show that $\preceq$ is a partial order.
	It suffices to show that $x \preceq y$ and $y \preceq z$ imply $x \preceq z$.
	Consider an apartment $\varSigma$ containing simplices $\{x,y\}$ and $\{y,z\}$.
	Regard $\varSigma = {\cal O}(\{0,1\}^n \setminus \{{\bf 0},{\bf 1}\})$, 
	where we can assume that $\ell(x) = \sum_i x_i$ since $\ell$ is natural.
	Then $x \leq y$ and $y \leq z$ in $\{0,1\}^n$.
	Hence $x \leq z$ holds in $\{0,1\}^n$.
	Consequently, $x$ and $z$ belong to a common simplex in $\varSigma \subseteq \varDelta$, implying $x \preceq z$.
	In particular, $\varDelta$ is the order complex of ${\cal L}(\varDelta) \setminus \{{\bar 0},{\bar 1}\}$.
	
	Next we show that ${\cal L}(\varDelta)$ is a lattice.
	Consider two vertices $x,y$.
	Suppose that $x$ and $y$ have  
	two different maximal common lower bounds $z,z' (\succ  \bar 0)$.
	Consider an apartment $\varSigma$ containing $\{x,z\}$ and $\{y,z'\}$.
	Regard $\varSigma = {\cal O}(\{0,1\}^n \setminus \{{\bf 0},{\bf 1}\})$.
	Then $z \leq \min (x, y) \geq z'$, and $\min (x, y)$ 
	is a common lower bound of $x$ and $y$.
	This contradicts the maximality of $z,z'$. 
	Thus the meet $x \wedge y$ exists, and is necessarily equal to 
	$\min (x, y)$ in this apartment. 
	Similarly the join $x \vee y$ exists, and is equal to $\max (x, y)$.
	In particular, we have
	$\ell(x) + \ell(y) = \sum_i x_i + \sum_i y_i = \sum_{i}\max (x,y)_i + \sum_i \min (x,y)_i = \ell(x \vee y) + \ell (x \wedge y)$,  
	where we let $\ell(\bar 0) := 0$ and $\ell(\bar 1) := n$.
	Thus $\ell$ is a valuation, and ${\cal L}(\varDelta)$ is a modular lattice.
	Also $\bar 1$ is the join of atoms (vertices having color $1$). Indeed, 
	$\bar 1 = {\bf 1} = \sum_{i} e_i = \bigvee_i e_i$ in any apartment.
\end{proof}

\section{Uniform modular lattices}\label{sec:uniform}
In this section, we introduce the concept of uniform modular lattices (carefully)
and establish the relation (Theorem~\ref{thm:main}) to affine buildings of type A.
The {\em ascending operator} of a lattice ${\cal L}$ is a map  $(\cdot)^+: {\cal L} \to {\cal L}$ defined by
\begin{equation}\label{eqn:ascending}
(x)^+ := \bigvee \{ y \in {\cal L} \mid \mbox{$y$ covers $x$} \} \quad (x \in {\cal L}).
\end{equation}
A modular lattice ${\cal L}$ is said to be {\em uniform} if 
the ascending operator $(\cdot)^+$ is defined (i.e.,
the right hand side of (\ref{eqn:ascending}) 
exists for all $x$) and is an automorphism on ${\cal L}$. 

\begin{Ex}\label{ex:Z^n} 
	As in $\{0,1\}^n$ (see Example~\ref{ex:01}), 
	$\ZZ^n$ becomes a lattice with respect to vector order $\leq$, 
	where $x \wedge y = \min(x,y)$ and $x \vee y = \max(x,y)$.
	 The component sum $x \mapsto \sum_{i=1}^{n}x_i$
	 is a valuation, and hence $\ZZ^n$ is a modular lattice (by Lemma~\ref{lem:valuation}).
	 The ascending operator is 
	 equal to $x \mapsto x + {\bf 1}$, which is clearly an automorphism.
	 Thus $\ZZ^n$ is a uniform modular lattice.
\end{Ex}

\begin{Ex}\label{ex:tree}
Let $T = (V,E)$ be an infinite tree with no vertex of degree one.
Regard $T$ as a bipartite graph $(V_0,V_1;E)$, 
where $V_0$ and $V_1$ denote two color classes.
Define $\ell:V \to \{0,1\}$ 
by $\ell(x) := 0$ if $x \in V_0$ and $\ell(x) := 1$ if $x \in V_1$.
Consider the directed graph on vertex set $V \times \ZZ$ 
such that an edge of head $(x,k)$ and tail $(x',k')$
is given if and only if $x$ and $x'$ are adjacent in $T$ 
and $\ell(x') + 2 k' = \ell(x) + 2 k + 1$
($\Leftrightarrow$ $k=k'$ and $\ell(x') = \ell(x)+ 1$ or $k'= k+1$ and $\ell(x) = \ell(x')+ 1$).
This graph is acyclic, and naturally induces 
a partial order on $V \times \ZZ$.
Let ${\cal L}$ denote the resulting poset.
For an infinite path $P$ in $T$,  
the subposet $V(P) \times \ZZ$ for vertex set of $V(P)$ of $P$ 
is isomorphic to $\ZZ^2$. 
The join and meet of two points $x,y$ exist in $\ZZ^2 = V(P) \times \ZZ$ 
for an infinite path $P$ containing $x,y$.
One can see that the function $(x,k) \mapsto \ell(x) + 2 k$ 
is a valuation on ${\cal L}$. Hence ${\cal L}$ is a modular lattice.
Since every vertex has at least two neighbors, 
the ascending operator coincides with $(x,k) \mapsto (x,k+1)$,   
which is clearly an automorphism on~${\cal L}$.
Thus ${\cal L}$ is a uniform modular lattice.	
\end{Ex}

\begin{Ex}[{See also \cite[Section 5.15]{Abels91}}]\label{ex:module}
Let $K$ be a field with a discrete valuation $v$, that is, 
a function $v:K \to \ZZ \cup \{\infty\}$ satisfying 
$v(xy) = v(x)+v(y)$, $v(x) + v(y) \geq \min (v(x),v(y))$, and 
$v(x) = 0$ $\Leftrightarrow$ $x=0$.
Let $R := \{ x \in K \mid v(x) \geq 0 \}$ 
be the associated valuation ring.
It is known that $R$ is a PID and 
has a unique maximal ideal $m = \{ x \in K \mid v(x) > 0\}$.
The ideal $m$ is generated by an element $t \in K$ (called a {\em uniformizer}).
Consider the $K$-vector space $K^n$, which is also regarded as an $R$-module.
Let ${\cal L}$ be the family of all 
free $R$-submodules of $K^n$ with rank $n$.
Such a module (called a {\em lattice} in the literature~\cite[Section 19]{Garrett}) is precisely  
an $R$-submodule generated by a $K$-linear basis $v_1,v_2,\ldots,v_n$ of $K^n$.
%
Regard ${\cal L}$ as a poset with respect to the inclusion order.	
Then ${\cal L}$ is a uniform modular lattice, 
where the ascending operator is given by ${\cal L} \ni L \mapsto t^{-1} L$. 
To see this fact, 
first note a general fact that 
the family of all submodules of any module becomes 
a modular lattice with $\wedge = \cap$ and $\vee = +$; 
indeed it is easy to see $L \cap (N+M) = L \cap N + M$ for $M \subseteq L$.
Therefore, to see that ${\cal L}$ is a modular lattice, 
it suffices to verify that $L,M \in {\cal L}$ 
implies $L \cap M, L + M \in {\cal L}$.
This is immediate from 
$L \cap M \subseteq L \subseteq L+M \subseteq t^{-k} (L \cap M) \subseteq t^{-k} L$
for large $k \in \ZZ$ and 
the fact that every $R$-submodule of a free $R$-module (with PID $R$) is free.
Define $h:{\cal L} \to \ZZ$ by
\[
h(L) := v(\det (v_1\ v_2\ \cdots\ v_n)) \quad (L \in {\cal L}), 
\]
where $L$ is generated by a basis $v_1,v_2,\ldots,v_n$ of $K^n$.
Then $h(L)$ is independent of the choice of the basis.
Observe that 
for $L,M \in {\cal L}$ with $L \subseteq M$ 
it holds $h(L)\geq h(M)$ and holds $h(L) > h(M)$ if and only if $L \neq M$.	
(In fact, $-h$ is a valuation of ${\cal L}$ in the sense of Section~\ref{subsec:modular}.)
From this (and discreteness of $v$), 
we see that ${\cal L}$ satisfies (F).
Also, if $L$ is covered by $M$, 
then $M = L + R t^{-1} v$ for some $v \in L$.
From this, we see that 
the ascending operator of ${\cal L}$ 
is given by $L \mapsto t^{-1}L$ and is obviously an automorphism.

A particular example of such a field $K$ 
is the field $F(t)$ of rational functions over a field $F$. 
The valuation $v$ is given by 
$v(p/q) := \deg p - \deg q$ with two polynomials $p,q$, 
where $\deg$ takes the minimum degree of a polynomial. 
\end{Ex}

Above examples actually provide representatives of affine buildings of type A.
In Section~\ref{subsec:U=>E}, we show that 
any uniform modular lattice ${\cal L}$ yields an affine building of type A.
In Section~\ref{subsec:E=>U}, we show the reverse construction.

\subsection{Uniform modular lattices $\Rightarrow$ affine buildings of type A}\label{subsec:U=>E}
Let ${\cal L}$ be a uniform modular lattice.
Let $(\cdot)^-:{\cal L} \to {\cal L}$ denote the inverse of 
the ascending operator $(\cdot)^+$.
\begin{Lem}
	The inverse $(\cdot)^-$ of $(\cdot)^+$ is given by
	\begin{equation}\label{eqn:x^-}
	(x)^- = \bigwedge \{ w \in {\cal L} \mid \mbox{$w$ is covered by $x$}\} \quad (x \in {\cal L}).
	\end{equation}
	In particular, the opposite $\check{\cal L}$ of ${\cal L}$ 
	is a uniform modular lattice.
\end{Lem}
\begin{proof}
	By definition, $(x)^+$ is the join of all atoms of $[x, (x)^+]$.
	Hence $[x, (x)^+]$ is a complemented modular lattice (by (cm1)). 
	We show that if $y \in {\cal L}$ is covered by $(x)^+$, 
	then $y$ belongs to $[x, (x)^+]$, i.e., $x \preceq y$.
	Indeed, since $(\cdot)^+$ is an automorphism, 
	there is $y' \in {\cal L}$ such that $(y')^+ = y$.
	Also $x$ covers $y'$,  
	which implies $x \preceq (y')^+$ by the definition of $(\cdot)^+$.
	The opposite of $[x, (x)^+]$ is also complemented modular (by (cm3)).
	Therefore $x$ is the meet of all elements ({\em coatoms}) 
	covered by $(x)^+$ in $[x, (x)^+]$. 
	By the above argument, they are exactly elements covered by $(x)^+$ in ${\cal L}$. 
    This means that the right hand side of (\ref{eqn:x^-}) exists, 
    and equal to $(x)^-$.  
\end{proof}
For an integer $k \in \ZZ$, let $(\cdot)^{+k}$
be defined as $((\cdot)^{+(k-1)})^+$ if $k > 0$, $((\cdot)^{+(k+1)})^-$ if $k < 0$, and 
the identity map if $k=0$. For $k > 0$, we also denote $(\cdot)^{+(-k)}$ by $(\cdot)^{-k}$.
\begin{Lem}\label{lem:u-rank}
	For $x,y \in {\cal L}$, 
	the intervals $[x, (x)^+]$ and $[y, (y)^+]$ are complemented modular lattices of the same rank.
\end{Lem}
\begin{proof}
	We show that $[x, (x)^+]$ and $[y, (y)^+]$ have the same rank.
	It suffices to consider the case where $y$ covers $x$ (by (F)).
	Since $(\cdot)^+$ is an automorphism, 
	$(y)^+$ covers $(x)^+$.
	Therefore we have $1 + r[y,(y)^+] = r[x, (y)^+] = r[x,(x)^+] + 1$ (by (JD)), which implies $r[x,(x)^+] = r[y, (y)^+]$.
\end{proof}
The {\em uniform-rank} of ${\cal L}$ is defined as the rank $r[x, (x)^+]$
of interval $[x,(x)^+]$ for $x \in {\cal L}$.
A chain $x^0 \prec x^1 \prec \cdots \prec x^m$ is said to be {\em short}
if $x^m \preceq (x^0)^+$.
Define an equivalence relation $\sim$ on 
${\cal L}$ by $x \sim y$ if $(x)^{+k} = y$ for some $k \in \ZZ$.
Let ${\cal C}({\cal L})$ 
be the simplicial complex on ${\cal L} /{\sim}$  
consisting of all short chains in ${\cal L}$ modulo $\sim$.
The goal of this section is to show the following.
\begin{Thm}\label{thm:main1}
	Let ${\cal L}$ be a uniform modular lattice of uniform-rank $n \geq 2$.
	Then the simplicial complex ${\cal C}({\cal L})$ is an affine building of type A with dimension $n-1$.
\end{Thm}
Let us return to the above Examples~\ref{ex:Z^n}, \ref{ex:tree}, and \ref{ex:module}.
For ${\cal L} = \ZZ^n$ (Example~\ref{ex:Z^n}), 
the simplicial complex ${\cal C}({\cal L})$
is nothing but the affine Coxeter complex of type A, 
since any maximal short chain is the form of (\ref{eqn:Coxeter_e}) 
and the ascending operator is $x \mapsto x+ {\bf 1}$.
In the case of Example~\ref{ex:tree}, 
${\cal C}({\cal L})$ is regarded as the original tree $T$.
It is well-known that 
an infinite tree without vertices of degree one
is a 1-dimensional affine building (of type A).
In Example~\ref{ex:module},
the complex ${\cal C}({\cal L})$ is 
nothing but the affine building for SL($K^n$).
This is a canonical example of an affine building of type A; 
see~\cite[Section 19]{Garrett}. 
Apartments are given by ${\cal C}({\cal L}(Q))$
for the sublattice ${\cal L}(Q)$ of ${\cal L}$ consisting of 
modules $R t^{\alpha_1} v_1 + R t^{\alpha_2} v_2+ \cdots R t^{\alpha_n}v_n$ 
for nonsingular $Q  = (v_1\ v_2\ \cdots\ v_n )\in K^{n \times n}$ and $\alpha \in \ZZ^n$.
Observe that ${\cal L}(Q)$ is isomorphic to the opposite of $\ZZ^n$ 
with $(x)^+ = x - {\bf 1}$ for $x \in {\cal L}(Q) = \ZZ^n$, 
and ${\cal C}({\cal L}(Q))$ is isomorphic to the affine Coxeter complex of type A.
The definition of uniform modular lattice 
is inspired by this example.

In the following, 
we suppose that the uniform-rank of ${\cal L}$ is equal to $n$.
Motivated by the above ${\cal L}(Q)$,  
define a {\em $\ZZ^n$-skeleton} of ${\cal L}$ by 
a sublattice ${\cal F}$ that is isomorphic to $\ZZ^n$ and satisfies
$(x)^+ = x + {\bf 1}$ for all $x \in {\cal F}$, 
where $x \mapsto x + {\bf 1}$ is the ascending operator in ${\cal F} = \ZZ^n$.
The proof of Theorem~\ref{thm:main1} goes along 
precisely the same line of the proof of Theorem~\ref{thm:spherical-1}.
Thus we show the following two lemmas. 
The first one corresponds to Lemma~\ref{lem:twochains}, and will be proved later.
\begin{Lem}\label{lem:twochains_e}
	For two short chains $C,D$ in ${\cal L}$, 
	there exists a $\ZZ^n$-skeleton of ${\cal L}$ containing $C,D$.
\end{Lem}

Let $C = (x = x^0 \prec x^1 \prec \cdots \prec x^n = (x)^+)$ be a maximal short chain.
Let us define 
the {\em relative position} $y_C \in \ZZ^n$ of an element $y \in {\cal L}$
with respect to $C$.
Choose a $\ZZ^n$-skeleton ${\cal F}$ containing $C$ and $y$ via Lemma~\ref{lem:twochains_e}.  
Identify ${\cal F}$ with $\ZZ^n$ 
so that $x^i - x^{i-1} = e_i$ for $i=1,2,\ldots,n$.
Define the relative position $y_C \in \ZZ^n$ 
as the integer vector $y - x$ in this coordinate.

 \begin{Lem}\label{lem:indep2}
 	The relative position $y_C$ of $y \in {\cal L}$ is independent of 
 	the choice of a $\ZZ^n$-skeleton containing~$C,y$.
 \end{Lem}
 \begin{proof}
 	We may assume that $x \preceq y$, since $((x)^-)_C  = x_C - {\bf 1}$ 
 	in any $\ZZ^n$-skeleton containing $x$.
 	Define sequence $x= z^0, z^1,\ldots, z^k = y$ in ${\cal F} = \ZZ^n$
 	with $k:= \max_i y_i- x_i$ by
 	\begin{equation}\label{eqn:z^j}
 	z^{j} := (z^{j-1} + {\bf 1}) \wedge y = z^{j-1} + 
 	\sum \{e_i \mid i: z^{j-1}_i < y_i \}.
 	\end{equation}
 	Since $z^j$ is obtained from $z^{j-1}, y$ by taking the ascending operator and $\wedge$,  
 	any $\ZZ^n$-skeleton containing $z^{j-1},y$ also contains $z^j$.
 	Consequently every $\ZZ^n$-skeleton containing $x,y$ 
 	contains the whole sequence $z^j$. 
 	Now $y_C = \sum_{j=1}^k  \sum \{ e_i \mid i: z^{j-1}_i < y_i \}$.
 	An index $i$ with $z^{j-1}_i < y_i$ is precisely an index 
 	with $(x^i + (j-1){\bf 1}) \wedge y \succ (x^{i-1} + (j-1){\bf 1}) \wedge y$.
 	This means that the indices of the sum in (\ref{eqn:z^j}) 
 	are independent of 
 	the choice of a $\ZZ^n$-skeleton.
    Thus the relative position $y_C$ is independent of the choice of a $\ZZ^n$-skeleton. 
 \end{proof}

Assuming the two lemmas, we complete 
the proof of Theorem~\ref{thm:main1}.
\begin{proof}[Proof of Theorem~\ref{thm:main1}]
	Short chains in a $\ZZ^n$-skeleton are short chains in ${\cal L}$ (by $(x)^+ = x+ {\bf 1}$). 
	Therefore $\ZZ^n$-skeletons induce subcomplexes in ${\cal C}({\cal L})$.
	We show that these subcomplexes satisfy the axiom of apartments.  
    Observe that they are isomorphic to the affine Coxeter complex of type A, which implies~(B1).
	Consider two simplices $A,B$ in ${\cal C}({\cal L})$, 
	which come from two short chains $C,D$
	in ${\cal L}$.
	By Lemma~\ref{lem:twochains_e} there is a $\ZZ^n$-skeleton containing $C,D$. 
	This implies (B2).
	Suppose that two $\ZZ^n$-skeletons ${\cal F},{\cal G}$ contain two short chains $C,D$. 
    Both ${\cal F}$ and ${\cal G}$ are regarded as $\ZZ^n$.
    To distinguish them, the unit vectors 
    of ${\cal F}$ and of ${\cal G}$ are denoted by 
    $e_1,e_2,\ldots,e_n$ and $e'_1,e'_2,\ldots,e'_n$, respectively.
    By appropriate renumbering and translation, 
    we can assume that 
    $C$ is equal to $({\bf 0} \leq e_1 \leq e_1+e_2 \leq \cdots \leq e_1+e_2 + \cdots + e_n = {\bf 1})$ in ${\cal F}$ and $({\bf 0} \leq e'_1 \leq e'_1+e'_2 \leq \cdots \leq e'_1+e'_2 + \cdots + e'_n = {\bf 1})$ in ${\cal G}$.
    Consider an isomorphism $\varphi: {\cal F} \to {\cal G}$ 
    defined by $\sum_{i=1}^n z_i e_i \mapsto \sum_{i=1}^n z_i e'_i$.
    The map $\varphi$ obviously induces a bijection between short chains.
    Moreover, by Lemma~\ref{lem:indep2}, $\varphi$ is the identity on
    the set ${\cal F} \cap {\cal G}$ of all common points.
    In particular, $\varphi$ is the identity on $C \cup D$.
    Hence $\varphi$ induces an isomorphism with (B3). 
\end{proof}

The remainder of this section is devoted to proving Lemma~\ref{lem:twochains_e}.
In the following, the rank $r[x,y]$ of interval $[x,y]$ is denoted by $r_x(y)$.
The function $y \mapsto r_x(y)$ is the rank function of 
the sublattice consisting of elements $y$ with $y \succeq x$.

We start with studying representations of $\ZZ^n$-skeletons.
A {\em segment} is a chain $a^0 \prec a^1 \prec \cdots \prec a^s$
such that $a^l$ covers $a^{l-1}$ for $l=1,2,\ldots,s$, 
and $a^{l+1} \not \in [a^{l-1}, (a^{l-1})^+] (\ni a^l)$ for $l=1,2,\ldots,s-1$.
A {\em ray} is an infinite chain $a^0 \prec a^1 \prec \cdots \prec a^l \prec \cdots$
satisfying this property for all $l = 1,2,\ldots.$
If $x = a^0$, 
a segment and a ray are called an {\em $x$-segment} and {\em $x$-ray}, respectively. 
\begin{Lem}\label{lem:segment}
	A segment in ${\cal L}$ is a segment in the opposite $\check{\cal L}$.
\end{Lem}
\begin{proof}
	$a^{l+1} \not \in [a^{l-1}, (a^{l-1})^+]$ implies 
	$a^{l+1} \wedge (a^{l-1})^+ = a^l$.
	Then $(a^{l+1})^- \wedge a^{l-1} = (a^l)^- \prec (a^{l+1})^-$.
	This implies that $a^{l-1} \not \succeq (a^{l+1})^-$, 
	and $a^{l-1} \not \in [(a^{l+1})^-,a^{l+1}]$.
	Hence $(a^l)$ is a segment in $\check{\cal L}$.
\end{proof}
\begin{Lem}
	Any $x$-segment can be extended to an $x$-ray.
\end{Lem}
\begin{proof}
Let $x = a^0 \prec a^1 \prec \cdots \prec a^s$ be an $x$-segment.
Since $(a^{s-1})^+$ is covered by $(a^s)^+$, and $[a^s, (a^s)^+]$
is complemented (by Lemma~\ref{lem:u-rank}), 
we can choose an atom $a^{s+1}$ 
in $[a^s, (a^s)^+]$ such that $a^{s+1} \vee (a^{s-1})^+ = (a^{s})^+$; in particular 
$a^{s+1} \not \in [a^{s-1}, (a^{s-1})^+]$.
Then $a^0 \prec a^1 \prec \cdots \prec a^s \prec a^{s+1}$ is a segment.
Repeating this process, we obtain a ray 
$a^0 \prec a^1 \prec \cdots \prec a^s \prec \cdots$.
\end{proof}
$x$-segments (or $x$-rays) $(x = a_i^0 \prec a_i^1 \prec a_i^2 \prec \cdots )$ $(i=1,2,\ldots,k)$ are 
said to be {\em independent} if $a_1^1,a_2^1,\ldots,a_k^1$ are 
independent atoms in $[x, (x)^+]$, i.e., 
$r_x(a_1^1 \vee a_2^1 \vee \cdots \vee a_k^1) = k$.
An ordered set $\alpha = (a_i^l)_{i,l}$ of $k$ independent $x$-segments 
is called a {\em partial $k$-frame} at $x$.
Let $\langle \alpha \rangle$ denote the sublattice generated by all $a_{i,l}$ in $\alpha$.
\begin{Lem}\label{lem:f_i^k}
	Let $\alpha = (a_i^l)_{i=1,2,\ldots,k, l=0,1,\ldots,s_i}$ be a partial $k$-frame at $x$.  
	For an element $p \in {\cal L}$ satisfying $p \wedge (\bigvee_i a_i^1) = x$, 
	define $b_i^l$ by
	\begin{equation}\label{eqn:f_i^k}
	b_i^l := p \vee a_i^l \quad (i=1,2,\ldots,k, l=0,1,2,\ldots,s_i).
	\end{equation}
	Then $\beta := (b_i^l)_{i,l}$ is a partial $k$-frame at $p$, 
	where the map $u \mapsto p \vee u$ is an isomorphism from $\langle \alpha \rangle$ to  $\langle \beta \rangle$ with $r[x,p] = r[u,p \vee u]$.	
\end{Lem}
\begin{proof}
	It suffices to prove the statement for 
	the case where $p$ covers $x$.
	We first show, by induction on $l$, 
	that $b_i^l$ covers $a_i^l$ and $b_i^{l-1}$ for any $i=1,2,\ldots,k$.
	Here we let $b^l := b_i^l$ and $a^l := a_i^l$ for simplicity.
    In the case of $l=1$, 
	this is true by $p \wedge (\bigvee_i a_i^1) = x = a^0$ (and equality~(\ref{eqn:valuation}) for $r_x$).
	Suppose that $l > 1$ and that $b^{l-1}$ covers $a^{l-1}$ and $b^{l-2}$.
	Necessarily $p = b^0 \not \preceq a^{l'}$ for $l' \leq l-1$.
	If $p \preceq a^l$, then it must hold $a^l = b^l = b^{l-1} = a^{l-1} \vee b^{l-2} 
	\in [a^{l-2}, (a^{l-2})^+]$; 
	this is a contradiction to $a^l \not \in [a^{l-2}, (a^{l-2})^+]$.
	Thus $b^l = p \vee a^l$ covers both of $a^l$ and $b^{l-1}$.
	
	Next we show that $b^{l+1} \not \in [b^{l-1}, (b^{l-1})^+]$.
	Now $b^l$ covers $a^l$ and $b^{l-1}$.
	Necessarily $(b^l)^+$ covers $(a^l)^+$ and $(b^{l-1})^+$; 
	in particular $(b^l)^+ = (a^l)^+ \vee (b^{l-1})^+$. 
	Also $a^{l+1} \vee (a^{l-1})^+ = (a^l)^+$  
	must hold (since $[a^{l-1}, (a^{l-1})^+] \not \ni a^{l+1} \in [a^{l},(a^l)^+]$). 
	By $a^{l+1} \preceq b^{l+1}  = b^{l} \vee a^{l+1} \preceq (a^l)^+$, 
	we have $(a^l)^+ = a^{l+1} \vee (a^{l-1})^+ \preceq b^{l+1} \vee (a^{l-1})^+ \preceq (a^l)^+$, implying $b^{l+1} \vee (a^{l-1})^+ =  (a^l)^+$.
	Therefore $(b^l)^+ = (a^l)^+ \vee (b^{l-1})^+ 
	= b^{l+1} \vee (a^{l-1})^+ \vee (b^{l-1})^+ = b^{l+1} \vee (b^{l-1})^+$.
	Since $(b^l)^+$ covers $(b^{l-1})^+$, 
	we have $b^{l+1} \not \preceq (b^{l-1})^+$, as required.
	
	Thus $(b_i^l)$ for each $i$ is a $p$-segment.
	The independence of $b_1^1, b_2^1, \ldots, b_k^1$ follows from
	$r_p(b_1^1 \vee b_2^1 \vee \cdots \vee b_k^1)
	=  r_x(a_1^1 \vee a_2^1 \vee \cdots \vee a_k^1 \vee p) - 1
	= r_x(p) + r_x(a_1^1 \vee a_2^1 \vee \cdots \vee a_k^1) - r_x(x) - 1 = k$.  
	
	If $u = a_1^{z_1} \vee a_2^{z_2} \vee \cdots \vee a_k^{z_k} \in \langle \alpha \rangle$, then $p \vee u =  
	b_1^{z_1} \vee b_2^{z_2} \vee \cdots \vee b_k^{z_k} \in \langle \beta \rangle$, and 
	the  statement for $u \mapsto p \vee u$ is an immediate consequence of 
	the next lemma (and (\ref{eqn:valuation})).
\end{proof}

Let $\alpha = (a_i^l)_{i=1,2,\ldots,k, l=0,1,\ldots,s_i}$ 
at a partial $k$-frame at $x$.
For an integer vector $z \in [{\bf 0}, s] \subseteq \ZZ^k$ with $s = (s_1,s_2,\ldots,s_k)$,
 define an element 
$\alpha(z) \in \langle \alpha \rangle$ by
\begin{equation}\label{eqn:map}
\alpha(z) := a_1^{z_1} \vee a_2^{z_2} \vee \cdots \vee a_k^{z_k}.
\end{equation}

\begin{Lem}\label{lem:ray}
	For a partial $k$-frame $\alpha = (a_i^l)_{i=1,2,\ldots,k,\, l=0,1,\ldots,s_i}$ 
	at $x$, 
	the sublattice $\langle \alpha \rangle$ 
	is isomorphic to $[{\bf 0}, s] \subseteq \ZZ^n$, where the map $z \mapsto \alpha(z)$ is an isomorphism from $[{\bf 0}, s]$ to $\langle \alpha \rangle$
	such that 
	\begin{equation}\label{eqn:l_m}
	r_x(\alpha(z)) = z_1 + z_2 + \cdots + z_k.
	\end{equation}
\end{Lem}

\begin{proof}
	We first show equation (\ref{eqn:l_m}) by induction on $k$. 
	In the case of $k = 1$, this is obvious. Suppose $k > 1$.
	Then $a^{z_k}_k \wedge (\bigvee_{1 \leq i \leq k-1} a_i^1) = x$ holds.
	Indeed, by Lemma~\ref{lem:indep} with $p = \bigvee_{1 \leq i \leq k-1} a_i^1$, 
	chain $(p \vee a_k^{l})_{l=0,1,2,\ldots,z_k}$ is a $p$-segment (of length $z_k$). 
	Thus $ r[x, a_k^{z_k}] = z_k = r[p, p \vee a_k^{z_k}] = r[p \wedge a_k^{z_k},a_k^{z_k}]$ and (rc) imply $a^{z_k}_k \wedge p= x$.
	Define a partial $(k-1)$-frame $\beta = (b_i^l)_{i,l}$ at $a_k^{z_k}$ 
	according to (\ref{eqn:f_i^k}) with $p = a_k^{z_k}$.
	Then 
	$\alpha(z)  = \beta (z')$ 
	and also $r_{x}(\alpha(z)) 
	= z_k + r_{p}(z')$, where $z'$ denotes the vector in $\ZZ^{k-1}$ obtained 
	from $z \in \ZZ^k$ by omitting the $k$-th coordinate $z_k$ of $z$.
	Since $\beta$ is a partial $(k-1)$-frame, by induction we have
	$r_{p}(z') = z_1+z_2+ \cdots + z_{k-1}$, from which we obtain (\ref{eqn:l_m}).
	
	Next we show
	\begin{equation}\label{eqn:distri}
	(a^{z_1}_1 \vee a^{z_2}_2 \vee \cdots \vee a^{z_{k-1}}_{k-1}) \wedge a^{z_{k}}_k = x
	\end{equation}
	Then ($\succeq$) is obvious.	
	Consider $r_x (a^{z_1}_1 \vee a^{z_2}_2 \vee \cdots \vee a^{z_{k-1}}_{k-1}) + r_x(a^{z_k}_k) = r_x((a^{z_1}_1 \vee a^{z_2}_2 \vee \cdots \vee a^{z_{k-1}}_{k-1}) \wedge a^{z_{k}}_k) + r_x(\alpha(z))$.
	By (\ref{eqn:l_m}) and (rc), we have ($=$).
	
	We are ready to prove the statement.
	By using (\ref{eqn:distri}), 
	every element $u$ in $\langle \alpha \rangle$ 
	can be written as
	$
	u = \alpha(z)
	$
	for some $z \in [{\bf 0},s]$.
	It suffices to show that this expression is unique.
	For each $i=1,2,\ldots,k$, 
	choose the maximum index $z_i' \in \ZZ_+$ 
	such that $a^{z_i'}_i \preceq u$.
	Then $z_i \leq z_i'$ (since $a^{z_i}_i \preceq u$).
	Consider $u' := \alpha(z')$. 
	Then $u' \preceq u$, implying $r_x(z') \leq r_x(z)$.
	On the other hand, $r_x(u) = z_1+ z_2 + \cdots + z_k \leq z_1' + z_2' + \cdots + z_k' = r_x(u')$. Thus, by (rc), it must hold $u = u'$ and $z_i = z_i'$ for $i=1,2,\ldots,k$.
\end{proof}
An ordered set of $n$ independent rays (at $x$)
is particularly called a {\em frame}.
For a frame $\alpha = (a_i^l)$, the sublattice $\langle \alpha \rangle$ 
is isomorphic to $\ZZ_+^n := (\ZZ \cap [0,\infty))^n$ 
so that $(\alpha(z))^+ = \alpha(z + {\bf 1})$.
\begin{Lem}\label{lem:Z^n}
	The sublattice of ${\cal L}$ is a $\ZZ^n$-skeleton if and only 
	if it is equal to $\bigcup_{k =0,1,2,\ldots}\langle \alpha \rangle^{-k}$ for a frame $\alpha$.
\end{Lem}
\begin{proof}
	Let ${\cal F}$ be a $\ZZ^n$-skeleton, and identified with $\ZZ^n$.
	Observe that the chain $(l e_i)_{l=0,1,2,\ldots}$
	is a ${\bf 0}$-ray. Then the set $\alpha = (l e_i)_{i,l}$ is a frame with 
	${\cal F} = \ZZ^n = \bigcup_{k=0,1,2,\ldots} (\ZZ^n_+ - k {\bf 1}) = \bigcup_{k=0,1,2,\ldots} \langle \alpha \rangle^{-k}$.
	Let $\alpha = (a_i^l)_{i,l}$ be a frame. 
	It is easy to see that  
	$\langle \alpha \rangle^{-k}$ is isomorphic to $(\ZZ \cap [-k, \infty))^n$ by 
	$z \mapsto (\alpha(z + k {\bf 1}))^{-k}$.
	Also it holds $\langle \alpha \rangle^{-k} \subseteq \langle \alpha \rangle^{-k-1}$
	since $\alpha(z)^{-k} = \alpha(z+{\bf 1})^{-k-1}$.
	Therefore $\bigcup_{k=0,1,2,\ldots}\langle \alpha \rangle^{-k}$ 
	is a $\ZZ^n$-skeleton.  
\end{proof}

\begin{Lem}\label{lem:preceq}
	For $x,y \in {\cal L}$, 
	there is $k \geq 0$ 
	such that $x \preceq (y)^{+k}$.
\end{Lem}
\begin{proof}
	We may assume that $x \not \preceq y$.
	Hence $x \succ x \wedge y$.
	Choose an atom $a$ in $[x \wedge y, x]$.
	By $a \wedge y = x \wedge y$ and modularity equality~(\ref{eqn:valuation}) for $r_{x \wedge y}$, 
	$a \vee y$ is an atom in $[y, (y)^+]$.
	Consequently $x \wedge y \prec a \preceq x \wedge (y)^+$.
	Thus, for $k \geq r[x \wedge y,x]$, it holds $x \wedge (y)^{+k} = x$, 
	implying $x \preceq (y)^{+k}$.
\end{proof}

\begin{proof}[Proof of Lemma~\ref{lem:twochains_e}]	We may assume that both $C$ and $D$ are maximal short chains (of length $n$).
Suppose that $C \subseteq [x,(x)^+]$ 
and $D \subseteq [y, (y)^+]$.
We may assume that 
$x \preceq y$. 
Indeed, for the general case, 
choose $k$ such that $x \preceq (y)^{+k}$ by Lemma~\ref{lem:preceq}.
Then $y' := (y)^{+k}$ satisfies $x \preceq y'$.
Any $\ZZ^n$-skeleton containing $C$ and $(D)^{+k} \subseteq [y', (y')^+]$, 
also contains $C$ and $D$. 

We first define elements $x_j,y_j,h_j$ and chains $C_j,D_j$, 
along an intuition in Figure~\ref{fig:fig1}.
	\begin{figure}[t]
		\begin{center}
			\includegraphics[scale=0.8]{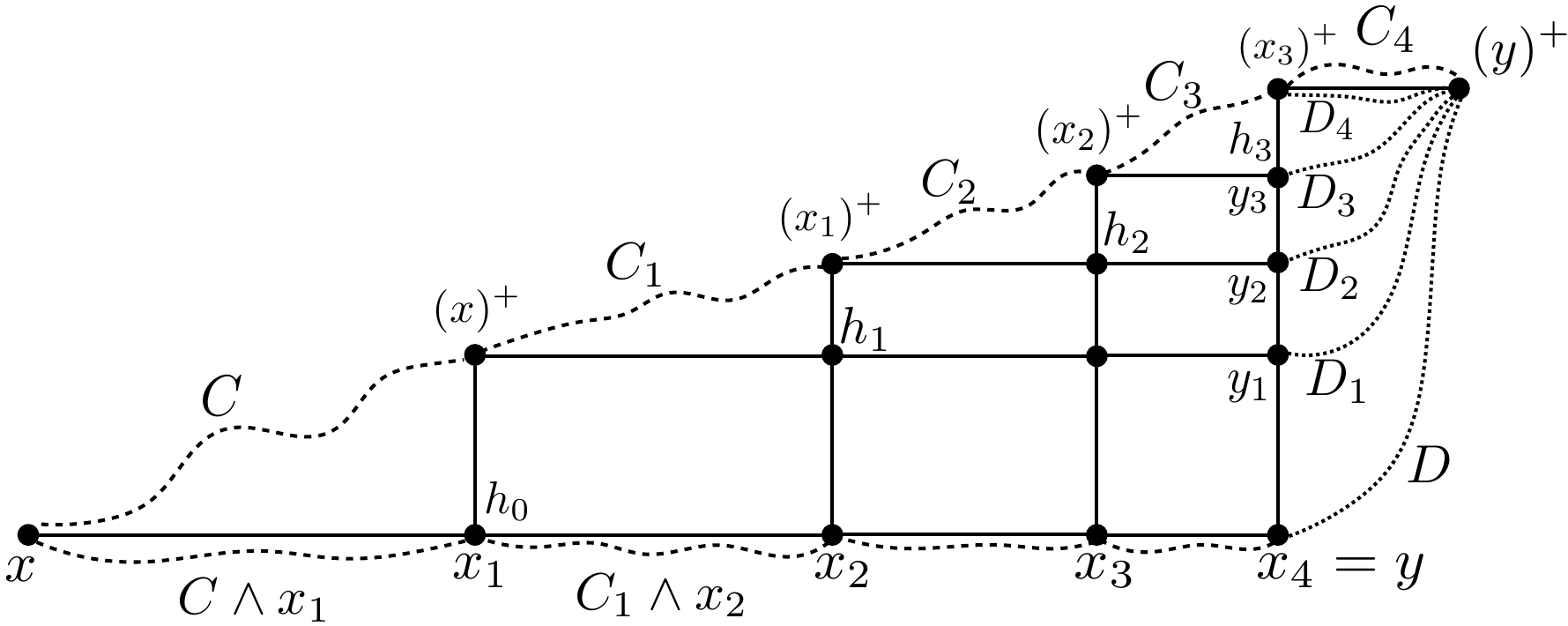}
			\caption{$x_j,y_j,h_j,C_j,D_j$}
			\label{fig:fig1}
		\end{center}
	\end{figure}\noindent
For $j=0,1,2,\ldots$, define $x_j$ by $x_0 := x$ and
\begin{eqnarray*}
x_j &:=& (x_{j-1})^+ \wedge y \\
&= &\bigvee \{a \mid \mbox{$a$ is $x_{j-1}$ or an atom in $[x_{j-1},y]$} \} \quad \in [x_{j-1},(x_{j-1})^+], 
\end{eqnarray*}
where the (second) equality follows from the observation 
that $(x_{j-1})^+ \wedge y$ belongs to complemented modular lattice 
$[x_{j-1}, (x_{j-1})^+]$ (by $x_{j-1} \preceq y$) 
and hence $(x_{j-1})^+ \wedge y$ is the join of atoms $a$ in 
$[x_{j-1}, (x_{j-1})^+]$ (with $a \preceq y$).
For some $m (\leq r[x,y])$, 
it holds $x = x_0 \prec x_1 \prec \cdots \prec x_m = x_{m+1} = \cdots = y$.
For $j=0,1,2,\ldots,m$, define short chain $C_j$ by $C_0 := C$ 
and
\begin{equation*}
C_j := (C_{j-1} \wedge x_j)^+. 
\end{equation*}
Since $C_{j-1} \wedge x_j$ is a chain in $[x_{j-1},x_j]$, 
$C_j$ is a chain in $[(x_{j-1})^+, (x_{j})^+]$.
Define $y_j$ and $D_j$ by $y_0= y$ and $D_0 = D$, and
\begin{eqnarray*}
y_j & := & (x_{j-1})^+ \vee y,  \\
D_j & := & D_{j-1} \vee y_j = D \vee y_j.
\end{eqnarray*}
Finally define $h_j$ by $h_0 := x_1$ and
\begin{equation}\label{eqn:h_j}
h_j := x_{j+1} \vee (x_{j-1})^+ = (x_{j})^+ \wedge y_j, 
\end{equation}
where the last equality follows from $x_{j+1} \vee (x_{j-1})^+ = ((x_j)^+ \wedge y) \vee (x_{j-1})^+ =  (x_j)^+ \wedge (y \vee (x_{j-1})^+) = (x_j)^+ \wedge y_j$.

Now $C_m$ and $D_m$ are maximal chains 
in complemented modular lattice $[(x_{m-1})^+, (y)^+]$.
By Lemma~\ref{lem:twochains}, there is a basis $a_1,a_2,\ldots,a_{k}$ 
of $[(x_{m-1})^+,(y)^+]$ 
such that $\langle a_1,a_2,\ldots,a_{k} \rangle$ contains $C_m$ and $D_m$.
In particular, $\alpha := ((x_{m-1})^+,a_i)_{i=1,2,\ldots k}$ is a partial $k$-frame at $(x_{m-1})^+$
such that $\langle \alpha \rangle$ contains $C_m$ and $D_m$.

We are going to show that a (given) partial $k$-frame 
$\alpha = (a_i^l)_{i=1,\ldots,k,\, l=0,1,\ldots,s_i,}$ 
at $(x_{j})^+$ with $ \langle \alpha \rangle \supseteq 
C_{j+1},D_{j+1}$ can be   
extended to a partial $k'$-frame $\beta = (b_i^l)_{i,l}$ at $(x_{j-1})^+$ 
with $\langle \beta \rangle \supseteq  C_{j}, D_{j}$, 
where we let $x_{-1} := (x)^{-}$ so that $(x_{-1})^+ = x$.
The case of $j=0$ is our goal.
\begin{figure}[t]
	\begin{center}
		\includegraphics[scale=0.8]{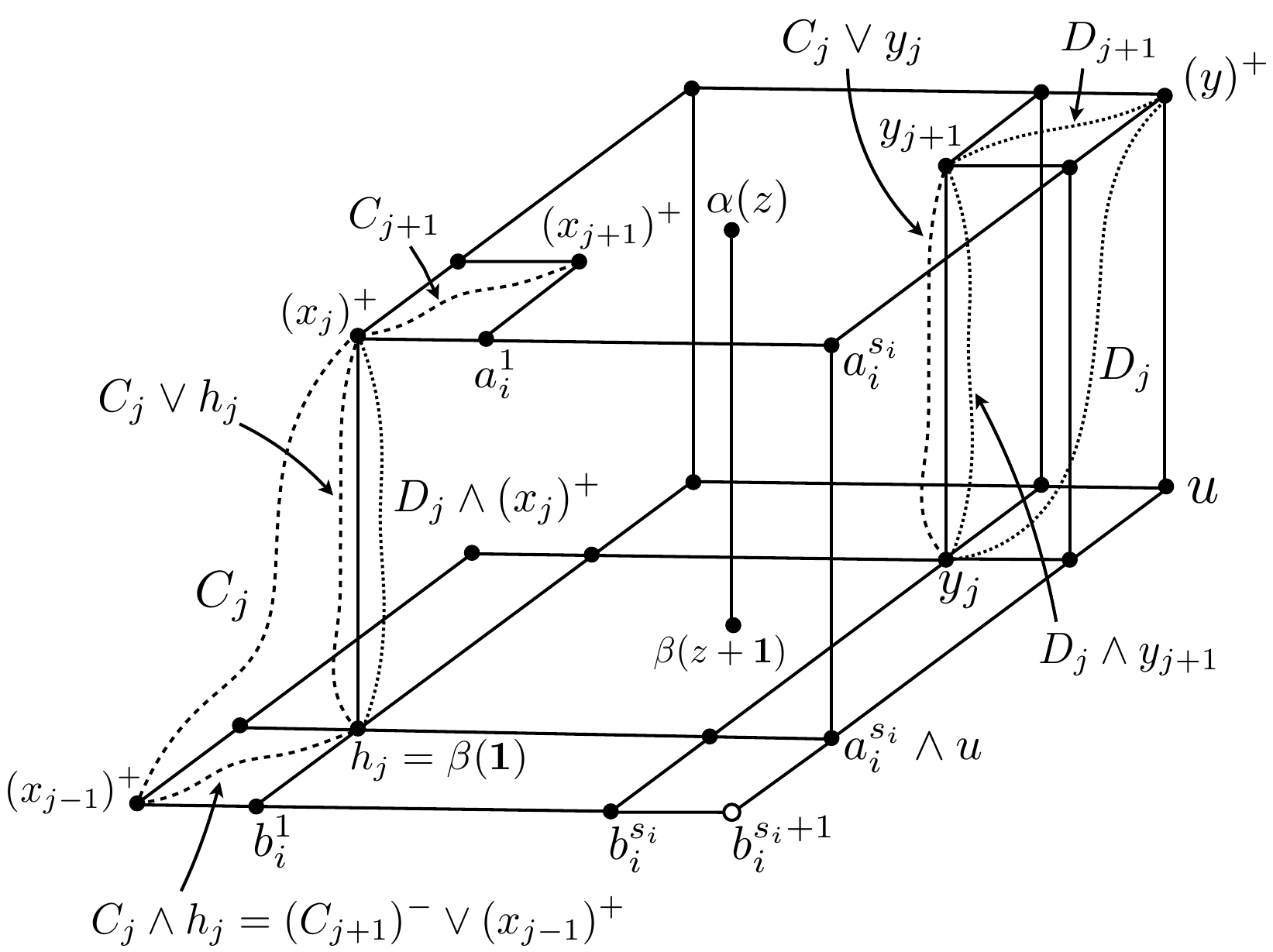}
		\caption{Extending a partial $k$-frame}
		\label{fig:fig2}
	\end{center}
\end{figure}\noindent
Figure~\ref{fig:fig2} illustrates an intuition 
behind the argument we proceed with. 
We assume that $\alpha({\bf 1}) = (x_{j+1})^+$ 
and $\alpha(s) = (y)^+$ for $s := (s_1,s_2,\ldots,s_k)$; 
this is true for the base case of the induction.
Let $b_i^l$ be defined by
\begin{equation*}
b_i^l := (x_{j-1})^+ \vee (a_i^l)^{-} \quad (i=1,2,\ldots,k, l=0,1,2,\ldots,s_i).
\end{equation*}
Obviously $((a_i^l)^{-})$ is a partial $k$-frame at $x_j$. 
By $(x_{j-1})^+ \wedge (\alpha({\bf 1}))^- = (x_{j-1})^+ \wedge x_{j+1} = x_j$ and Lemma~\ref{lem:f_i^k},  
$\beta := (b_i^l)_{i,l}$ is a partial $k$-frame 
at $(x_{j-1})^+$. Here
$\langle \beta \rangle$ contains 
$C_{j} \wedge h_{j}$ since 
$\langle \beta \rangle \supseteq (x_{j-1})^+ \vee (C_{j+1})^- = (x_{j-1})^+ \vee (C_{j} \wedge x_{j+1})
= C_{j} \wedge (x_{j+1} \vee (x_{j-1})^+) = C_{j} \wedge h_{j}$; 
the last equality follows from~(\ref{eqn:h_j}).
Also $\beta(s) = (x_{j-1})^+ \vee (\alpha(s))^- = (x_{j-1})^+ \vee ((y)^+)^- = (x_{j-1})^+ \vee y = y_{j}$.
Choose any complement $u$ of $y_{j+1}$ in $[y_{j},(y)^+] \subseteq [y,(y)^+]$
such that $\langle [y_{j},y_{j+1}] \cup [y_{j},u] \rangle$ contains $D_{j}$ 
(via (cm2) and Lemma~\ref{lem:complement}~(1)). 
Then $u$ is also a complement of $(x_j)^+$ in $[h_j,(y)^+]$.
Indeed, $y_{j+1} = (x_j)^+ \vee y = (x_j)^+ \vee y_{j}$ implies that
$(x_j)^+ \vee u = (y)^+$. By this fact with (\ref{eqn:valuation}) for $r_{h_j}$, 
$r[u,(y)^+] = r[y_j,y_{j+1}] = r[h_j,(x_j)^+]$ (by (\ref{eqn:h_j}))
and the rank-comparison argument (rc) for $h_j \preceq (x_j)^+ \wedge u$, 
we conclude $(x_j)^+ \wedge u = h_j$.
\begin{Clm} For $z \in [{\bf 0},s]$, the following hold:
	\begin{itemize}
		\item[(1)] 
		$\beta(z+ e_i) \preceq  \alpha(z) \wedge u$ if $z_i < s_i$.
		\item[(2)] $r[\beta(z), \alpha(z) \wedge u] = k$.
	\end{itemize}
	In particular, $\beta(z+{\bf 1}) = \alpha(z) \wedge u$ 
	for  $z \in [{\bf 0}, s-{\bf 1}]$.
\end{Clm}
\begin{proof}
	We first show $\beta(z) \preceq \alpha(z) \wedge u$.
	Obviously $\beta(z) \preceq \beta(s) = y_j \prec u$. 
	We show $\beta(z) \preceq \alpha(z)$.
	Since $(x_{j-1})^{+}$ is the join of atoms in $[x_j,(x_j)^+]$ 
	and $(\alpha(z))^-$ is a complement of $(x_{j-1})^{+}$ in $[x_j, \beta(z)]$ 
	(by Lemmas~\ref{lem:f_i^k} and \ref{lem:ray}), 
	$\beta(z)$ is also the join of atoms in $[\alpha(z)^-, \alpha(z)]$ (by Lemma~\ref{lem:isomorphism}~(1)), implying $\beta(z) \preceq \alpha(z)$ with $r[x_j, (x_{j-1})^+] = r[(\alpha(z))^-, \beta(z)]$.
	
	Here $u$ is a complement of $(x_j)^+$ in $[h_j, (y)^+]$.
	This implies that 
	$r[h_j,(x_j)^+] = r[\alpha(z) \wedge u, \alpha(z)]$ (via Lemma~\ref{lem:isomorphism}).
	Now $n = r[x_j,(x_j)^+] = r[x_j,(x_{j-1})^+] + r[(x_{j-1})^+,h_j]+ r[h_j,(x_j)^+]$ and 
	$n = r[(\alpha(z))^-,\alpha(z)] = r[(\alpha(z))^-, \beta(z)] +r[\beta(z),\alpha(z) \wedge u] +r[\alpha(z) \wedge u, \alpha(z)]$ (by (JD) and Lemma~\ref{lem:u-rank}).
	Hence 
	we have (2) $k = r[(x_{j-1})^+,h_j] = r[\beta(z), \alpha(z) \wedge u]$.
	Suppose $z_i < s_i$. Then 
	$\beta(z+e_i) := (x_{i-1})^{+} \vee (\alpha(z+e_i))^- = \beta(z) \vee (\alpha(z+e_i))^-$. 
	Here $\beta(z)$ is the join of atoms in $[(\alpha(z))^-,\alpha(z)]$ 
	(as seen above), and
	$(\alpha(z+e_i))^-$ is an atom of $[\alpha(z)^-,\alpha(z)]$.
	This means that $\beta(z+ e_i) \preceq \alpha(z)$.
	Thus we have (1). 	Now $\beta(z+{\bf 1}) = \alpha(z) \wedge u$ 
	for  $z \in [{\bf 0}, s-{\bf 1}]$ 
	follows from the combination of (2),  
	(rc) for $\beta(z+{\bf 1}) \preceq \alpha(z) \wedge u$ (by (1)), and Lemma~\ref{lem:ray}.
\end{proof}

By this claim, 
for each $i =1,2,\ldots,k$, 
$\beta({\bf 1} + (s_i-1)e_i)$ is covered by $a_i^{s_i} \wedge u = \alpha(s_i e_i) \wedge u$. 
Since $[b_i^{s_i}, a_i^{s_i}] \subseteq [(a_i^{s_i})^-, a_i^{s_i}]$
is complemented modular, 
we can choose a complement $b_i^{s_i+1}$ of 
$\beta({\bf 1} + (s_i-1)e_i)$ in 
$[b_i^{s_i}, a_i^{s_i} \wedge u]$, 
which covers $b_i^{s_i}$; see Figure~\ref{fig:fig2}.

We show that $(b_i^{l})_{l=0,1,\ldots, s_{i}+1}$ is a segment.
By Lemmas~\ref{lem:segment} and \ref{lem:f_i^k} 
for the opposite $\check{\cal L}$, 
we have that 
$(a_i^{s_i-2},a_i^{s_i-1},a_i^{s_i}) \wedge u$ is a segment.
By the above claim, we must have 
$(a_i^{s_i-2},a_i^{s_i-1},a_i^{s_i}) \wedge u = 
(\beta({\bf 1}+ (s_i - 2) s_i), \beta({\bf 1}+ (s_i - 1) s_i), a^{s_i}_i \wedge u)$.
Similarly, 
$(b_i^{s_i-1},b_i^{s_i},b_i^{s_i+1}) = ((\beta({\bf 1}+ (s_i - 2) e_i), \beta({\bf 1}+ (s_i - 1) e_i), \alpha(s_i e_i)) \wedge u)) \wedge b_i^{s_i+1}$ is a segment.  
This concludes that $(b_i^{l})_{l=0,1,\ldots, s_{i}+1}$ is a segment.

Now $\beta = (b_i^{l})_{i=1,2,\ldots,k,\, l=0,1,\ldots,s_i+1}$ 
is a partial $k$-frame at $(x_{j-1})^+$ 
such that 
$\beta({\bf 1}) = h_j$ and $\beta(s + {\bf 1}) = u$.
By Lemma~\ref{lem:f_i^k}, the above claim, $b_i^{s_i+1} \preceq u$, and (rc), 
we must have $\beta(z+{\bf 1}) = \alpha(z) \wedge u$ for all $z \in [{\bf 0},s]$, 
which in turn implies $\beta(z+{\bf 1}) \vee (x_j)^+ = \alpha(z)$ (by Lemma~\ref{lem:isomorphism}~(1)). 
Consider two short chains $C_{j} \vee y_{j}$ and $D_{j} \wedge y_j$ in $[y_{j},y_{j+1}]$. Choose a basis $c_1,c_2,\ldots,c_{i'}$ of 
$[y_{j}, y_{j+1}]$ that generates 
$C_{j} \vee y_{j}$ and $D_{j} \wedge y_{j+1}$.
By Lemma~\ref{lem:complement}, 
the sublattice 
$\langle c_1,c_2,\ldots,c_{i'}, \beta(s+e_1), \beta(s+e_2),\ldots, \beta(s+e_k) \rangle$
contains $D_j$.
Also $c_1 \wedge (x_j)^+, c_2 \wedge (x_j)^+, \ldots,c_{i'} \wedge (x_j)^+$
is a basis of $[h_j, (x_j)^+]$ (by Lemma~\ref{lem:isomorphism} (1)) that 
generates $C_{j} \vee h_j$ and $D_{j} \wedge (x_j)^+$.
Choose a complement $v$ of $h_j$ in 
$[(x_{j-1})^+, (x_j)^+]$
such that $\langle [(x_{j-1})^+, v] \cup [(x_{j-1})^+, h_j] \rangle$ 
contains $C_{j}$ (via Lemma~\ref{lem:complement}).
Here $c_1 \wedge v, c_2 \wedge v,\ldots, c_{i'} \wedge v$ 
is a basis of $[(x_{j-1})^+, v]$ (by Lemma~\ref{lem:isomorphism}~(1)).
Append segment $((x_{j-1})^+, c_i \wedge v)$ to $\beta$ for $i=1,2,\ldots,i'$.
Then we obtain a partial $(k+i')$-frame $\beta$ at $(x_{j-1})^+$ 
such that $\langle \beta \rangle$ contains $C_{j}$ and $D_{j}$ 
(since $(c_i \wedge v) \vee \beta(s) = c_i$), 
and $\beta({\bf 1}) = (x_j)^+$ and $\beta(s+{\bf 1}) = (y)^+$, as required.

For $j=0$, we obtain a partial $n$-frame $\alpha$ at $x$ that generates $C$ and $D$.
Extend each $x$-segment in $\alpha$ 
to an $x$-ray according to Lemma~\ref{lem:ray}.
Thus we obtain a frame $\alpha$ that generates $C$ and $D$, 
and obtain a $\ZZ^n$-skeleton $\bigcup_{k=0,1,2,\ldots} \langle \alpha \rangle^{-k}$ 
containing $C$ and $D$ (Lemma~\ref{lem:Z^n}). 
This completes the proof of Lemma~\ref{lem:twochains_e}.
\end{proof}

\subsection{Affine buildings of type A $\Rightarrow$ uniform modular lattices}\label{subsec:E=>U}

Let $\varDelta$ be an affine building of type A with dimension $n-1$.
We first introduce a special labeling analogous to a natural coloring in the spherical case.
Consider the subposet $\varLambda$ of $\ZZ^n$ defined by
\begin{equation}
\varLambda := \left\{ x \in \ZZ^n \bigmid 0 \leq \sum_{i=1}^{n} x_i \leq n-1 \right\}.
\end{equation}
For every point $x \in \ZZ^n$, 
there is a unique $x' \in \varLambda$ with $x - x' \in \ZZ{\bf 1}$.
Namely $\varLambda$ is the set of representatives of $\ZZ^n / \RR{\bf 1}$.
Consider the order complex ${\cal O}(\varLambda)$ of $\varLambda$.
Then one can observe that the affine Coxeter complex of type A is isomorphic 
to the subcomplex ${\cal O}'(\varLambda)$ of ${\cal O}(\varLambda)$  
consisting of chains $x_0 < x_1 < \cdots < x_m$
with $x_m < x_0 + {\bf 1}$.

A coloring 
$\ell: \varDelta^0 \to \{0,1,2,\ldots,n-1\}$ of $\varDelta$ is said to be {\em natural}
if for every apartment $\varSigma$
there is an isomorphism $\varrho: \varSigma \to {\cal O}'(\varLambda)$
with $\ell(x) = \sum_{i=1}^n \varrho(x)_i$ for all $x \in \varSigma^0$.
\begin{Lem}\label{lem:natural}
	A natural coloring exists.
\end{Lem}
From the group-theoretic view, 
a natural coloring corresponds to 
the ordering of generators of 
the affine Coxeter group of type A so 
that consecutive numbers are assigned to adjacent generators 
in the Dynkin diagram of type $\tilde A_n$ (that is a cycle). 

\begin{proof}
	Before constructing a natural coloring, 
	we note one remark on automorphisms on ${\cal O}(\varLambda)$.
	For $z \in \ZZ^n$, 
	define $\kappa^{\pm}_z: \varLambda \to \varLambda$ 
	so that $\kappa^{\pm}_z(x)$ is the unique point in $\varLambda$ with
	$\pm x + z - \kappa^{\pm}_z(x) \in \RR {\bf 1}$.
	Then $\kappa^{\pm}_z$ induces an automorphism on ${\cal O}'(\varLambda)$ 
	with $\sum_i \kappa^\pm_z(x)_i =  \sum_i (\pm x_i +  z_i) \mod n$.
	
	Let us start the proof of this lemma.
	Consider an arbitrary apartment $\varSigma$.
	Choose an arbitrary isomorphism $\varrho: \varSigma \to {\cal O}'(\varLambda)$.
	Define the color $\ell(x)$ of 
	$x \in \varSigma^0$ by $\ell(x) := \sum_{i=1}^n \varrho(x)_i \in \{0,1,\ldots,n-1\}$.
	This is a coloring of $\varSigma$.
	Extend this coloring to the whole 
	coloring $\ell: \varDelta^0 \to \{0,1,\ldots,n-1\}$; see (col) in Section~\ref{subsec:building}.
	We verify that $\ell$ is indeed natural. 
    Here we observe:
	\begin{itemize}
		\item[($*$)] In $\varSigma$, every vertex $x$ is 
		adjacent to ${n \choose |k - \ell(x)|}$ vertices of color $k \neq \ell(x)$.
	\end{itemize}
	Consider another apartment $\varSigma'$.
	Suppose first that $\varSigma$ and $\varSigma'$ 
	have a common maximal simplex $A$.
	Choose an isomorphism $\varSigma \to \varSigma'$ fixing $A$ via (B3).
	This isomorphism is taken to be color-preserving by (B3$^+$). 
	Therefore
	the property ($*$) holds in $\varSigma'$.
	Consider an isomorphism 
	$\varrho':\varSigma' \to {\cal O}'(\varLambda)$.
	Now $\ell': {\varSigma'}^0 \to \{0,1,\ldots,n-1\}$ 
	defined by $x \mapsto \sum_{i=1}^n \varrho'(x)_i$ is a coloring on $\varSigma'$.
	By replacing $\ell'$ by $\kappa^{\pm}_v \circ \ell'$ if necessarily, 
	we can assume that 
	$\ell'(x) = 0$ for the vertex $x \in A$ with $\ell (x) = 0$ and 
	$\ell'(y) \neq n-1$ for the vertex $y \in A$ with $\ell(y) = 1$.
	In $\varSigma'$, 
	each vertex $x$ must satisfy ($*$) for $\ell'$. 
	By the uniqueness of coloring on $\varSigma'$, 
	there is a bijection $\sigma$ on $\{0,1,\ldots,n-1\}$ 
	such that $\ell'  = \sigma \circ \ell$.
	Then it must hold $\sigma(0) = 0$ and $\sigma(1) \neq n-1$.
	All vertices in $A$ satisfy ($*$) for $\ell$ and for $\ell'$.
	Therefore ${n \choose |k - l|} = {n \choose |\sigma(k) - \sigma(l)|}$ 
	must hold for $k \neq l$.
	For $|k-l| = 1$ or $n-1$, it holds $n ={n \choose |\sigma(k) - \sigma(l)|}$, 
	which implies $|\sigma(k) - \sigma(l)| = 1$ or $n-1$.
	By $\sigma(0) = 0$ and $\sigma(1) \neq n-1$, 
	it holds $\sigma(1) = 1$, consequently, $\sigma(2) = 2$, $\sigma(3) = 3,\ldots$.
	Thus $\sigma$ is the identity, and $\varrho'$ is a desired isomorphism.

	Next suppose that $\varSigma'$ is arbitrary.
	By (B3), 
	there is an apartment $\varSigma''$ containing 
	a maximal simplex in $\varSigma$ and a maximal simplex $B$ in $\varSigma''$.
	Apply the above argument with replacing $\varSigma'$ 
	by $\varSigma''$ and $A$ by $B$.
	Then we obtain a desired isomorphism from $\varSigma'$ to 
	${\cal O}'(\varLambda)$.
\end{proof}

Fix a natural coloring $\ell$.
We construct a uniform modular lattice from $\varDelta$.
Our construction generalizes that in Example~\ref{ex:tree}.
Consider a directed graph $G(\varDelta)$ on vertex set $\varDelta^0 \times \ZZ$, 
where two distinct vertices $(x,k)$ and $(x',k')$ 
have an edge from $(x',k')$ to $(x,k)$, denoted by $(x',k') \to (x,k)$, if
$x$ and $x'$ belong to a common simplex in $\varDelta$ 
and $\ell(x') + k'n = \ell(x) + kn + 1$ 
($\Leftrightarrow$ either $k'= k$ and $\ell(x') = \ell(x) +1$ 
or $k' = k+1$ and $\ell(x) = \ell(x') + n-1$).
The graph $G(\varDelta)$ is acyclic,  
since $(x,k) \mapsto \ell(x) + n k$ 
is monotone decreasing on any directed path.
Define a partial order $\preceq$ on $\varDelta^0 \times \ZZ$ by
$(x,k) \preceq (x',k')$ if there is a directed path on 
$G(\varDelta)$ from $(x',k')$ to $(x,k)$.
The resulting poset on $\varDelta^0 \times \ZZ$ satisfies (F), 
and is denoted by ${\cal L}(\varDelta)$.
\begin{Thm}\label{thm:main2}
Let $\varDelta$ be an affine building of type A.	
Then ${\cal L}(\varDelta)$ is a uniform modular lattice, 
where $\varDelta$ is equal to ${\cal C}({\cal L}(\varDelta))$. 
\end{Thm}
The rest of this section is devoted to proving this theorem. 
%
Now each apartment $\varSigma$ can be regarded as ${\cal O}'(\varLambda)$, 
where $\ell(x) = \sum_{i=1}^n x_i$ for $x \in \varSigma^0 = \varLambda$.
Let ${\cal L}(\varSigma)$
denote the subposet of ${\cal L}(\varDelta)$ consisting 
of $(x,k) \in \varSigma^0 \times \ZZ$.
Consider the canonical retraction 
$\rho_{\varSigma} = \rho_{\varSigma,A}: \varDelta \to \varSigma$ (for some $A \in \varSigma$); see (ret) in Section~\ref{subsec:building}.
From this,  
we define an order-preserving retraction $\bar \rho_{\varSigma} :{\cal L}(\varDelta) \to {\cal L}(\varSigma)$ 
by $(x,k) \mapsto (\rho_{\varSigma}(x),k)$.
Under identification $\varSigma^0 = \varLambda \subseteq \ZZ^n$, we have:
\begin{Lem}\label{lem:apart}
	${\cal L}(\varSigma)$ is isomorphic to $\ZZ^n$ by $(x,k) \mapsto x + k{\bf 1}$.
\end{Lem}
\begin{proof}
	From ${\cal L}(\varSigma) = \varSigma^0 \times \ZZ = \varLambda \times \ZZ$, 
	we can easily see that the map is a bijection.
	In particular, we can identify ${\cal L}(\varSigma)$ with $\ZZ^n$.
	We need to show that  
	$p \preceq q$ if and only if $p \leq q$ in ${\cal L}(\varSigma) = \ZZ^n$.
	If $p \leq q$, 
	then there is a directed path from $q$ to $p$ in 
	the Hasse diagram of $\ZZ^n$; it is a directed path in $G(\varDelta)$, 
	implying $p \preceq q$.
	Suppose that $p \preceq q$.
	There is a directed path $P$ from $q$ to $p$ in $G(\varDelta)$.
	Then the image of $P$ by retraction $\bar \rho_{\varSigma}:{\cal L}(\varDelta) \to {\cal L}(\varSigma)$ is a directed path from 
	$q$ to $p$ in ${\cal L}(\varSigma) = \ZZ^n$.
	This means that $p \leq q$, as required.
\end{proof}

As mentioned in Section~\ref{subsec:building}, 
the geometric realization $|\varSigma|$ of an apartment $\varSigma$
is naturally regarded as $\RR^n / \RR {\bf 1}$ 
with $\varSigma^0 = \ZZ^n / \RR {\bf 1}$.
The next lemma is crucial for showing 
the existence of the meet and join in ${\cal L}(\varDelta)$.
\begin{Lem}\label{lem:every}
	For $x,y \in \ZZ^n = {\cal L}(\varSigma)$, let 
	$(\overline{x}^k)_k$ and $(\underline{x}^k)_k$ be sequences of points in $\ZZ^n$ 
	defined by
	\begin{eqnarray*}
	\overline{x}^{k} &:= &\overline{x}^{k-1} + \sum \{ e_i \mid i: y_i > \overline{x}^{k-1}_i \}, \\
	\underline{x}^{k} &:= &\underline{x}^{k-1} - \sum \{ e_i \mid i: y_i < \underline{x}^{k-1}_i \} \quad (k=1,2,\ldots), 
	\end{eqnarray*}
	where $\overline{x}^0 = \underline{x}^0 := x$, and 
	$\overline{x}^k = \max (x,y)$ and $\underline{x}^k = \min (x,y)$ for large $k$.
	Then the sequences $(\overline{x}^k + \RR{\bf 1})_k$ and $(\underline{x}^k+ \RR{\bf 1})_k$ of vertices in $\varSigma^0$
	belong to every apartment containing $x + \RR{\bf 1}$ and $y+ \RR{\bf 1}$.
\end{Lem}
\begin{proof}
We first show a property of 
the triangulation $\varSigma$ of $\RR^n / \RR {\bf 1}$.
\begin{Clm}
	For a vector $u \in \RR^n$, 
	the simplex of $|\varSigma|$ containing $u+ \RR {\bf 1}$ in its relative interior
	contains $\lceil u \rceil + \RR{\bf 1}$ and $\lfloor u \rfloor + \RR{\bf 1}$ as vertices.
\end{Clm}
Here $\lceil u \rceil$ (resp. $\lfloor u \rfloor$) is 
the integral vector obtained from $u$ by rounding up (resp. down) the fractional part of each component of $u$.
\begin{proof}
	Consider  
	the unique expression 
	$u - \lfloor u \rfloor = \sum_{k=1}^{n} \lambda_k \sum_{i \in X_k} e_i$
	for $\emptyset \neq X_1 \subset X_2 \subset \cdots \subset X_{n-1} \subset X_n = \{1,2,\ldots,n\}$ and $\lambda_i \geq 0$ with $\sum_{i=1}^{n} \lambda_i < 1$.	
	From this, we see that $u$ is a convex combination of $\lfloor u \rfloor$
	and $\lfloor u \rfloor + \sum_{i \in X_k} e_i$ for $k$ with $\lambda_k > 0$, 
	which forms the simplex in $|\varSigma|$ containing $u$ as its relative interior.
	Notice that $\lfloor u \rfloor + \sum_{i \in X_k} e_i$ for largest $k$ with $\lambda_k > 0$
	is equal to $\lceil u \rceil$. 
	This proves the claim.
\end{proof}
We prove the lemma. We use the CAT(0)-metrization of $|\varDelta|$; see Section~\ref{subsec:building}. 
The unique geodesic between $\overline{x}^k + \RR{\bf 1}$ 
and $y + \RR{\bf 1}$ in $|\varDelta|$ is given by $t \mapsto  (1-t) \overline{x}^k + t y + \RR{\bf 1} \in \RR^n / \RR{\bf 1} = |\varSigma|$.
By the uniqueness (geo) of the geodesic, 
every simplex meeting 
the geodesic in its relative interior
must belong to every apartment containing 
$\overline{x}^k+ \RR{\bf 1}$ and $y + \RR{\bf 1}$.
Notice that
for small $\epsilon > 0$, 
the point $\lceil \overline{x}^k + \epsilon (y - \overline{x}^k) \rceil$ is equal to $\overline{x}^{k+1}$.
By the claim, $\overline{x}^{k+1} + \RR{\bf 1}$ is a vertex of a simplex 
with which the geodesic meets in its relative interior.
This means that $\overline{x}^{k+1} + \RR{\bf 1}$ belongs to every apartment containing 
$\overline{x}^k+ \RR{\bf 1}$ and $y+ \RR{\bf 1}$.
Consequently, the whole sequence $(\overline{x}^k + \RR{\bf 1})$
belongs to every apartment containing $x+ \RR{\bf 1}$ and $y+ \RR{\bf 1}$.
The statement for $\underline{x}^k$ is shown by replacing $\overline{x}^k$ 
and $\lceil \cdot \rceil$ with  $\underline{x}^k$ and $\lceil \cdot \rceil$, respectively.
\end{proof}

\begin{proof}[Proof of Theorem~\ref{thm:main2}]
We first show that ${\cal L}(\varDelta)$ is a lattice.
Consider any two elements $(x,k)$ and $(y,l)$ of ${\cal L}(\varDelta)$. 
By (B2), 
there is an apartment $\varSigma$ containing $x$ and $y$.
By Lemma~\ref{lem:apart}, 
$(x,k)$ and $(y,l)$ are regarded as integer vectors 
$p = x + k{\bf 1}$ and $q = y + l{\bf 1}$, respectively.
Thus, in ${\cal L}(\varDelta) = \ZZ^n$, we can consider the meet 
$p \wedge_{\varSigma} q := \min (p,q)$ and the join 
$p \vee_{\varSigma} q := \max (p,q)$.
We show that $p \wedge_{\varSigma} q$ and $p \vee_{\varSigma} q$
are independent of the choice of the apartment $\varSigma$.
Consider another apartment $\varSigma'$ containing 
$p+ \RR{\bf 1}$ and $q + \RR{\bf 1}$.
By Lemma~\ref{lem:every}, 
$\varSigma'$ contains $p \vee_\varSigma q + \RR{\bf 1}$, and
hence ${\cal L}(\varSigma')$ contains $p \vee_\varSigma q$.
Conversely, ${\cal L}(\varSigma)$ contains $p \vee_{\varSigma'} q$.
Consider the order-preserving retraction 
$\bar \rho_{\varSigma}$. 
We have $p \vee_{\varSigma} q \preceq \bar \rho_{\varSigma}(p \vee_{\varSigma'} q) = p \vee_{\varSigma'} q$.
Also, by considering $\bar \rho_{\varSigma'}$, we have $p \vee_{\varSigma} q \preceq p \vee_{\varSigma'} q$.
Thus $p \vee_{\varSigma} q = p \vee_{\varSigma'} q$, and 
operator $\vee := \vee_{\varSigma}$ is independent of an apartment.
Similarly, $\wedge := \wedge_{\varSigma}$ 
is well-defined.
We show that $p \wedge q$ indeed equals the meet of $p$ and $q$.
Consider any common lower bound $u$ of $p$ and $q$.
We prove $p \wedge q \succeq u$ by the induction 
on the (minimum) length $k$ of a directed path $P$ from $p$ to $u$.
In the case of $k=0$, we have $u = p = p \wedge q$.
Suppose that $k> 0$. Consider 
the next element $p'$ following $p$ in $P$. 
Here $u$ is also a common lower bound of $p'$ and $q$.
By induction $p' \wedge q \succeq u$.
Also $p + \RR{\bf 1}$ and $p'+ \RR{\bf 1}$ belong to a common simplex.
There is an apartment $\varSigma$ such that ${\cal L}(\varSigma)$ 
contains $p,p'$ and $q$.
Then ${\cal L}(\varSigma) (= \ZZ^n)$ also contains $p \wedge q$ and $p' \wedge q$ 
with $p \wedge q \succeq p' \wedge q$, which implies $p \wedge q \succeq u$, as required.
By the same argument, $p \vee q$ is the join of $p$ and~$q$.

We show that ${\cal L}(\varDelta)$ is a modular lattice. 
Define $v:{\cal L}(\varDelta) \to \ZZ$ by $(x,k) \mapsto \ell(x) + kn$.
Then $(x,k) \prec (x',k')$ implies $v(x,k) < v(x',k')$.
For any apartment $\varSigma$, 
sublattice ${\cal L}(\varSigma)$ is identified with $\ZZ^n$, as shown above.
In this identification, any element $(x,k) \in {\cal L}(\varSigma)$ is 
regarded as $p = x + k{\bf 1} \in \ZZ^n$. By $\ell(x) = \sum_{i} x_i$, we have
$v(p) = \ell(x) + kn = \sum_{i} (x+ k {\bf 1})_i = \sum_{i}p_i$.
Hence modular equality (\ref{eqn:valuation}) holds on 
${\cal L}(\varSigma)$ for any apartment $\varSigma$, and 
holds on the whole ${\cal L}(\varDelta)$ by (B1).
Thus $v$ is a valuation on ${\cal L}(\varDelta)$, 
and ${\cal L}(\varDelta)$ is a modular lattice (by Lemma~\ref{lem:valuation}).

We finally verify 
that ${\cal L}(\varDelta)$ is uniform.
We show that the ascending operator $(\cdot)^+$ coincides with 
the map $(x,k) \mapsto (x,k+1)$, which is obviously 
an automorphism on ${\cal L}(\varDelta)$. 
If $(x',k')$ covers $(x,k)$, 
then $(x',k') \preceq (x,k+1)$ holds.
Therefore it suffices to show that $(x,k+1)$ 
is the join of some elements that covers $(x,k)$.
Consider ${\cal L}(\varSigma)$ containing $(x,k)$, 
which also contains $(x,k+1)$.
By Lemma~\ref{lem:apart}, 
$(x,k)$ and $(x,k+1)$ are regarded as integer vectors 
$x+ k{\bf 1}$ and $x + (k+1){\bf 1}$, respectively.
Then $x + (k+1){\bf 1}$ ($ = (x,k+1)$) is the join of 
$x+ k{\bf 1} + e_i$ for $i=1,2,\ldots,n$ (that covers $x+k{\bf 1}$).
Thus the ascending operator equals the map $(x,k) \mapsto (x,k+1)$, and 
we conclude that ${\cal L}(\varDelta)$ is a uniform modular lattice 
with $\varDelta = {\cal C}({\cal L}(\varDelta))$. 
\end{proof}

\section{Concluding remarks}\label{sec:concluding}
We close this paper with a few remarks. 
\paragraph{Modular graphs and affine buildings.}
In~\cite{CCHO}, we explored interesting connections 
between CAT(0)-spaces and 
various subclasses of {\em weakly modular graphs}. 
Among them, 
{\em orientable modular graphs} 
form a fascinating subclass of weakly modular graphs. 
They are defined as connected undirected graphs $G = (V,E)$
satisfying: 
\begin{itemize}
	\item For any triple of vertices $x_1,x_2,x_3 \in V$ 
	there is a vertex $y \in V$ such that 
	$d(x_i,x_j) = d(x_i,y) + d(y,x_j)$ for $1 \leq i < j \leq 3$, 
	where $d$ is the graph metric on $V$.
	\item There is an edge-orientation such that 
	every 4-cycle $(x_0,x_1,x_2,x_3)$ is oriented 
	as $x_i \to x_{i+1}$ if and only if $x_{i+2} \leftarrow x_{i+3}$.	 
\end{itemize}
(A graph satisfying the first condition is called a {\em modular graph}.)
It is shown in \cite[Section 6.8]{CCHO} 
that an {\em affine building $\varDelta$ of type C}, 
which also becomes a CAT(0)-space,  
gives rise to an orientable modular graph $G$
as a certain subgraph of the $1$-skeleton of $\varDelta$,  
in which the graph $G$ recovers original $\varDelta$ completely.
This raises a natural question: 
{\em do other affine buildings admit such 
a graph-theoretic approach by orientable or more generally weakly modular graph?}
The presented result may be an answer of this question for type A, 
since the (undirected) Hasse diagram of a modular lattice is 
an orientable modular graph.

\paragraph{L-convex functions on uniform modular lattices.}
The primary motivation of uniform modular lattices 
comes from a recent movement~\cite{HH16ext, HH16L-convex,HH17survey} of
{\em Discrete Convex Analysis beyond $\ZZ^n$}.
Originally {\em Discrete Convex Analysis (DCA)}~\cite{MurotaBook} 
was a theory of ``convex" functions on $\ZZ^n$ 
generalizing matroids and submodular functions in combinatorial optimization. 
In DCA, 
{\em L-convex functions} constitute one of fundamental 
classes of discrete convex functions on $\ZZ^n$. 
They are defined as functions $g:\ZZ^{n} \to \RR \cup \{\infty\}$ 
that satisfy the submodularity inequality
\begin{equation}
g(x) + g(y) \geq g(\min(x,y)) + g(\max(x,y)) \quad (x,y \in \ZZ^n),
\end{equation}
and satisfy the linearity over ${\bf 1}$-direction 
\begin{equation}
g(x+ k {\bf 1}) = g(x) + k \alpha  \quad (x \in \ZZ^n, k \in \ZZ)
\end{equation}
for some $\alpha \in \RR$.
Recent work~\cite{HH16ext, HH16L-convex} 
shows that analogues of L-convex functions
are definable on certain grid-like structures generalizing $\ZZ^n$ 
and bring meaningful applications
to several combinatorial optimization problems with which 
the previous DCA could not deal. 
In particular, \cite{HH16L-convex} introduces 
{\em L-convex functions on an affine building of type C}, 
and links them to the design of efficient algorithms 
for classes of network optimization problems; see also \cite{HH17survey}.

The concept of uniform modular lattices enables 
us to define what should be called 
{\em L-convex functions on an affine building of type A}.
Recall Example~\ref{ex:Z^n} that $\ZZ^n$ is a uniform modular lattice 
with ascending operator $x \mapsto x+{\bf 1}$.
Then the above definition of the L-convexity is naturally extended to 
an arbitrary uniform modular lattice ${\cal L}$.
A function $g:{\cal L} \to \RR \cup \{\infty\}$
is called {\em L-convex} if it satisfies the submodularity inequality
\begin{equation}
g(x) + g(y) \geq g(x \wedge y) + g(x \vee y) \quad (x,y \in {\cal L}),
\end{equation}
and satisfies the linearity on the ascending operator
\begin{equation}
g((x)^{+k}) = g(x) + k \alpha  \quad (x \in {\cal L}, k \in \ZZ)
\end{equation}
for some $\alpha \in \RR$.
In the case of $\alpha = 0$, an L-convex function $g$ is viewed as 
the vertex set ${\cal L}/{\sim}$ of 
the affine building ${\cal C}({\cal L})$ of type A.

In the subsequent paper~\cite{HH18b}, we link, via the affine building for SL$(\RR(t)^n)$ (Example~\ref{ex:module}),  
this new L-convex function to computation of the degree of 
the determinants of polynomial matrices; 
it is well-known in the literature 
that the $\deg$-$\det$ computation of polynomial matrices 
generalizes and abstracts a number of basic combinatorial optimization problems; 
see~\cite{MurotaMatrix}.

\section*{Acknowledgments}
The author thanks Yuni Iwamasa and Koyo Hayashi for careful reading and helpful comments, 
and thanks the referee for comments.
The work was partially supported by JSPS KAKENHI Grant Numbers 25280004, 26330023, 26280004, 17K00029.

\end{document}